\documentclass[a4paper,12pt]{amsart}

\usepackage{amssymb}

\newtheorem{theorem}{Theorem}[section]
\newtheorem{corollary}[theorem]{Corollary}

\newtheorem{prop}[theorem]{Proposition}

\def\A{\operatorname{A}}

\def\D{\operatorname{D}}
\def\E{\operatorname{E}}
\def\F{\operatorname{F}}

\def\U{\operatorname{U}}

\def\SU{\operatorname{SU}}

\def\Sp{\operatorname{Sp}}

\def\Ad{\operatorname{Ad}}
\def\ad{\operatorname{ad}}

\def\Aut{\operatorname{Aut}}

\def\dim{\operatorname{dim}}

\def\exp{\operatorname{exp}}
\def\Exp{\operatorname{Exp}}

\def\Int{\operatorname{Int}}

\def\Lie{\operatorname{Lie}}

\def\Out{\operatorname{Out}}
\def\Pin{\operatorname{Pin}}

\def\span{\operatorname{span}}
\def\Spin{\operatorname{Spin}}

\def\Sym{\operatorname{Sym}}

\def\tr{\operatorname{tr}}

\newcommand{\fre}{\mathfrak{e}}
\newcommand{\frf}{\mathfrak{f}}
\newcommand{\frg}{\mathfrak{g}}
\newcommand{\frh}{\mathfrak{h}}

\newcommand{\frk}{\mathfrak{k}}
\newcommand{\frl}{\mathfrak{l}}

\newcommand{\frp}{\mathfrak{p}}

\newcommand{\frs}{\mathfrak{s}}
\newcommand{\frt}{\mathfrak{t}}
\newcommand{\fru}{\mathfrak{u}}

\newcommand{\frz}{\mathfrak{z}}

\newcommand{\bbC}{\mathbb{C}}

\newcommand{\bbR}{\mathbb{R}}
\newcommand{\bbZ}{\mathbb{Z}}
\newcommand{\bbH}{\mathbb{H}}

\begin{document}

\title[Klein Four subgroups of Lie Algebra Automorphisms]{Klein Four subgroups of Lie Algebra Automorphisms}

\date{February 2011}

\author{Jing-Song Huang}
\address[Huang]{Department of Mathematics, Hong Kong University of Science and Technology,
Clear Water Bay, Kowloon, Hong Kong SAR, China}
\email{mahuang@ust.hk}

\author{Jun Yu}
\address[Yu]{Department of Mathematics, ETH Z\"urich, 8092 Z\"urich,
Switzerland} \email{jun.yu@math.ethz.ch}


\abstract{By calculating the symmetric subgroups $\Aut(\fru_0)^{\theta}$ and their involution classes,  
we classify the Klein four subgroups $\Gamma$ of $\Aut(\fru_0)$ for each compact
simple Lie algebra $\fru_0$ up to conjugation. This leads to a new approach of classification
of semisimple symmetric pairs and $\bbZ_2\times \bbZ_2$-symmetric spaces. We also
determine the fixed point subgroup $\Aut(\fru_0)^\Gamma$.}

\endabstract



\maketitle

\section{Introduction}

The Riemannian symmetric pairs were classified by Cartan [C] and the more general semisimple
symmetric pairs by Berger [B]. The algebraic structure of semisimple
symmetric spaces is even interesting for geometric and analytic reasons. Some of the
recent work are Oshima and Sekiguchi's classification of the reduced
root systems \cite{Oshima-Sekiguchi} and Helminck's classification
for algebraic groups \cite{Helminck}.  Mostly recently some new approach
of classification and parametrization of semisimple symmetric pairs
were given in \cite{Huang} by using admissible quadruplets and in
\cite{Chuah-Huang} by using double Vogan diagrams.

In this paper we study the structure of semisimple symmetric space by
determining the Klein four subgroups in Lie algebra automorphisms.
Let $\fru_0$ be a simple compact Lie algebra and $\frg$ be its complexification.
Denote by $\Aut(\fru_0)$ the automorphism group of $\fru_0$. For any involution
$\theta$ in $\Aut(\fru_0)$, we first determine the centralizer $\Aut(\fru_0)^{\theta}$ of $\theta$,
which is a symmetric subgroup. By understanding the conjugacy classes of involutions in 
$\Aut(\fru_0)^{\theta}$, we proceed to classify Klein four subgroups in $\Aut(\fru_0)$ up to conjugation.
This gives a new approach for classification of commuting pairs of involutions of $\fru_0$
or $\frg$.  We note that the ordered commuting pairs of involutions corresponding
to Berger's classification of semisimple symmetric spaces.

If $\Gamma$ is a finite abelian subgroup of the automorphism group
of a Lie group $G$,  then the homogeneous space $G/H$ is called $\Gamma$-symmetric
provided $G^\Gamma_0\subseteq H\subseteq G^\Gamma$ (cf. [L]).  In the case $\Gamma=\bbZ_2$,
it is the symmetric space and in the case $\Gamma=\bbZ_k$ it is $k$-symmetric
space studied in \cite{Wolf-Gray}.  In the case $\Gamma=\bbZ_2\times \bbZ_2$ is the
Klein four group,  the exceptional $\bbZ_2\times \bbZ_2$-symmetric
spaces were studied in \cite{Kollross}.  This paper contains a complete list of all
$\bbZ_2\times \bbZ_2$-symmetric pairs.  Finally, we determine the fixed point subgroup
$\Aut(\fru_0)^\Gamma$.

\section{Preliminaries and notations}\label{Notations and Preliminary}

\subsection{Complex semi-simple Lie algebras} \label{complex
semi-simple algebra}

Let $\frg$ be a complex semisimple Lie algebra  and $\frh$
a Cartan subalgebra. Then  it has a root-space
decomposition
\[\frg=\frh \oplus(\bigoplus _{\alpha \in \Delta} \frg_{\alpha}),\]
where $\Delta=\Delta(\frg,\frh)$ is the root system of $\frg$ and
$\frg_{\alpha}$ is the root space of the root $\alpha\in\Delta$.
Let $B$ be the Killing form on $\frg$.
It is a non-degenerate symmetric form.  The restriction of $B$ to $\frh$ is also
non-degenerate. For any $\lambda \in \frh^{\ast}$, let
$H_{\lambda}\in \frh$ be determined by
\[B(H_{\lambda}, H)=\lambda(H), \forall H \in \frh.\] For any
$\lambda, \mu \in \frh^{\ast}$, define $\langle \lambda,
\mu\rangle:=B(H_{\lambda},H_{\mu})$. Note that, $\forall \alpha,\beta \in \Delta$,
\begin{eqnarray*}&&\langle
\alpha,\beta\rangle=B(H_{\alpha},H_{\beta})=\beta(H_{\alpha})=\alpha(H_{\beta})\in
\bbR,\\&& \langle\alpha,\alpha\rangle=B( H_{\alpha},H_{\alpha})=\alpha
(H_{\alpha}) \not= 0,\end{eqnarray*} and $2\langle\alpha,
\beta\rangle/\langle\beta, \beta\rangle \in \bbZ$. Let
\[H'_{\alpha}=\frac{2}{\alpha(H_{\alpha})}H_{\alpha}, A_{\alpha,
\beta}=2\langle\alpha, \beta\rangle/\langle\beta,
\beta\rangle=\alpha(H'_{\beta}).\] \
Then $[H'_{\alpha},
X_{\beta}]=\beta(H_{\alpha})X_{\beta}=2\langle\alpha,
\beta\rangle/\langle\alpha, \alpha\rangle=A_{\beta, \alpha}
X_{\beta}.$  The $\{H'_{\alpha}\}$ correspond to the co-roots
$\{\alpha^\vee=\frac{2\alpha}{\langle\alpha,\alpha\rangle}|\alpha\in\Delta\}$.

Choose a lexicography order to get a positive system $\Delta^{+}$
and a simple system $\Pi$. Drawing $A_{\alpha,\beta} A_{\beta,\alpha}$
edges to connect any two distinct simple roots $\alpha$ and $\beta$
defines the Dynkin diagram. We follow Bourbaki numbering to order the simple roots.
For simplicity,  we write $H_{i},H'_{i}$ for $H_{\alpha_{i}},H'_{\alpha_{i}}$ for a simple root $\alpha_{i}$.

Let $\Aut(\frg)$ be the group of all complex linear automorphisms
of $\frg$ and $\Int(\frg)$ the subgroup of inner automorphisms. We define
\[\Out(\frg):=\Aut(\frg)/\Int(\frg)\cong \Aut(\Pi),\]
and the exponential map $\exp: \frg\longrightarrow \Aut(\frg)$ by
\[\exp(X)=\exp(\ad(X)), \forall X\in \frg=\Lie(\Aut(\frg)).\]

\subsection{A compact real form} \label{a specific compact form}

One can normalize the root vectors $\{X_{\alpha}, X_{-\alpha}\}$ so that $\langle
X_{\alpha}, X_{-\alpha}\rangle=2/\alpha(H_{\alpha})$ . Then
$[X_{\alpha}, X_{-\alpha}]=H'_\alpha$. Moreover, one can normalize
$\{X_{\alpha}\}$ appropriately, such that
\[\fru_0=\span\{X_{\alpha}-X_{-\alpha}, i(X_{\alpha}+X_{-\alpha}), i H_{\alpha}:
\alpha \in \Delta\}\] is a compact real form of $\frg$ (\cite{Knapp}
pages $348-354$). Define \[\theta(X+iY):=X-iY ,\forall X,Y\in
\fru_0.\] Then $\theta$ is a Cartan involution of $\frg$ (as a real
Lie algebra) and $\fru_0=\frg^{\theta}$ is a maximal compact
subalgebra of $\frg$.  All real form of $\frg$ are conjugate by
$\Aut(\frg)$.

Let $\Aut(\fru_0)$ be the group of automorphisms of $\fru_0$
and $\Int(\fru_0)$ the subgroup of  inner automorphisms.
We define \[\Theta(f):=\theta f \theta^{-1}, \forall f\in \Aut(\frg).\]
Then it is a Cartan involution of $\Aut(\frg)$ with differential
$\theta$.  It follows that $\Aut(\fru_0)=\Aut(\frg)^{\Theta}$
and $\Int(\fru_0)=\Int(\frg)^{\Theta}$
are maximal compact subgroups of $\Aut(\frg)$ and $\Int(\frg)$ respectively.
We also set
\[\Out(\fru_0):=\Aut(\fru_0)/\Int(\fru_0)\cong \Out(\frg)\cong\Aut(\Pi).\]

\subsection{Notations and involutions}\label{involution representative}

We denote by $\fre_6$ the compact simple Lie algebra of type $\bf E_6$.  Let $\E_6$ be 
the connected and simply connected Lie group with Lie algebra $\fre_6$. Naturally, 
$\fre_6(\bbC)$ and $\E_6(\bbC)$ are their complexifications. Similar
notations will be used for other types.

Let $Z(G)$ denote the center of a group $G$ and $G_0$ denote the connected component of $G$ 
containing identity element. For $H\subset G$ ($\frh\subset\frg$),
let $Z_{G}(H)$ ($Z_{\frg}(\frh)$) denote the centralizer of $H$ in $G$ ($\frh$ in $\frg$), 
let $N_{G}(H)$ denote the normalizer of $H$ in $G$.

In the case $G=\E_6,\E_7$, let $c$ denotes a non-trivial element in $Z(G)$.

In the case $\fru_0=\fre_7$, let $H'_0=\frac{H'_2+H'_5+H'_7}{2}\in\fre_7$.

\smallskip

Let $\Pin(n)$ ($\Spin(n)$) be the Pin (Spin) group in degree $n$.
Write $c=e_1e_2\cdots e_{n}\in \Pin(n)$. Then $c$ is in $\Spin(n)$ if and only if $n$ is even.
In this case $c\in Z(\Spin(n))$.  If $n$ is odd, then $\Spin(n)$ has a Spinor module $M$ of
dimension $2^{\frac{n-1}{2}}$.  If $n$ is even, then $\Spin(n)$ has two Spinor
modules $M_{+},M_{-}$  of dimension $2^{\frac{n-2}{2}}$. We distinguish
$M_{+},M_{-}$ by requiring that  $c$ acts on $M_{+}$  and $M_{-}$ by scalar 1 or -1
respectively when $4|n$; and by $-i$ or $i$ respectively when $4|n-2$.

\smallskip

We need the following notations for the rest of the paper.

We first define the following matrices, 
  \[ J_{m}=\left(
\begin{array}{cc} 0&I_{m}\\ -I_{m}&0\\ \end{array} \right),
I_{p,q}=\left(\begin{array}{cc} -I_{p}&0\\ 0&I_{q}\\ \end{array}
\right),\] \[I'_{p,q}=\left( \begin{array}{cccc} -I_{p}&0&0&0\\
0&I_{q}&0&0\\ 0&0&-I_{p}&0\\ 0&0&0&I_{q}\\\end{array}  \right),\  
J_{p,q}=\left( \begin{array}{cccc} 0&I_{p}&0&0\\-I_{p}&0&0&0\\ 0&0&0&I_{q}\\ 0&0&-I_{q}&0\\\end{array}  \right),\]
\[K_{p}=\left( \begin{array}{cccc} 0&0&0&I_{p}\\0&0&-I_{p}&0\\ 0&I_{q}&0&0\\ -I_{q}&0&0&0\\\end{array}\right).\]

Then we define the following groups,
\[ Z_{m}=\{\lambda I_{m}|\lambda^{m}=1\},\] 
\[ Z'=\{(\epsilon_1,\epsilon_2,\epsilon_3,\epsilon_4)|
\epsilon_{i}=\pm{1}, \epsilon_1\epsilon_2\epsilon_3\epsilon_4=1\},\]
\[ \Gamma_{p,q,r,s}=\langle \left(
\begin{array} {cccc} -I_{p}&0&0&0 \\ 0& -I_{q}&0&0 \\ 0&0&I_{r}&0
\\ 0&0&0&I_{s}\\  \end{array}  \right), \left( \begin{array}{cccc}-I_{p}&0&0&0
\\0&I_{q}&0&0 \\ 0&0&-I_{r}&0 \\ 0&0&0&I_{s} \\  \end{array}
\right) \rangle.\]

\smallskip

We note that the conjugacy classes of involutions in the automorphism groups of
compact simple Lie algebras are in one-one correspondence with isomorphism classes of
real forms of complex simple Lie algebras, and in one-one correspondence with irreducible
Riemannian symmetric pairs $(\fru_0,\frh_0)$ (with $\fru_0$ compact simple) or 
$(\frg_0,\frh_0)$ (with $\frg_0$ non-compact simple). These are classified by
\'Elie Cartan in 1926. We list this classification here (cf.
\cite{Knapp}).

\smallskip
The classical compact simple Lie algebras are defined as follows.
For $F=\bbR,\bbC,\bbH$, let $M_{n}(F)$ be the set of $n\times n$
matrices with entries in $F$, and
\begin{eqnarray*}
&&\mathfrak{so}(n)=\{X\in M_{n}(\bbR)|X+X^{t}=0\},\\&&
\mathfrak{su}(n)=\{X\in M_{n}(\bbC)|X+X^{\ast}=0, \tr X=0\}, \\&&
\mathfrak{sp}(n)=\{X\in M_{n}(\bbH)|X+X^{\ast}=0\}.
\end{eqnarray*} Then
$\{\mathfrak{su}(n): n\geq 3\}$, $\{\mathfrak{so}(2n+1): n\geq 1\}$,
$\{\mathfrak{sp}(n): n\geq 3\}$, $\{\mathfrak{so}(2n): n\geq 4\}$
represent all compact classical simple Lie algebras.
\smallskip

We define the following involutions which are representatives for
all conjugacy classes of involutions.

i) Type $\bf A$. For $\fru_0=\mathfrak{su}(n),n\geq 3$, $\{\Ad(I_{p,n-p})| 1\leq p\leq \frac{n}{2}\}$
(type {\bf AIII}), $\{\tau\}$ (type {\bf AI}), $\{\tau\circ\Ad(J_{\frac{n}{2}})\}\}$ (type {\bf AII})
represent all conjugacy classes of involutions in $\Aut(\fru_0)$,
the corresponding real forms are
$\mathfrak{su}(p,n-p),\mathfrak{sl}(n,\bbR),\mathfrak{sl}(\frac{n}{2},\bbH)$.

ii) Type $\bf B$. For $\fru_0=\mathfrak{so}(2n+1),n\geq 1$, $\{\Ad(I_{p,2n+1-p})|
1\leq p\leq n\}$ (type {\bf BI}) represent all conjugacy classes of involutions in
$\Aut(\fru_0)$, and the corresponding real forms are
$\mathfrak{so}(p,2n+1-p)$.

iii) Type $\bf C$. For $\fru_0=\mathfrak{sp}(n),n\geq 3$, $\{\Ad(I_{p,n-p})| 1\leq
p\leq \frac{n}{2}\}$ (type {\bf CII}) and $\{\Ad(\textbf{i}I)\}$ (type {\bf CI}) represent all
conjugacy classes of involutions in $\Aut(\fru_0)$, and the
corresponding real forms are
$\mathfrak{sp}(p,n-p),\mathfrak{sp}(n,\bbR)$.

iv) Type $\bf D$. For $\fru_0=\mathfrak{so}(2n),n\geq 4$, $\{\Ad(I_{p,2n-p})| 1\leq
p\leq n\}$ (type {\bf DI}) and $\{\Ad(J_{n})\}$ (type {\bf DIII}) represent all conjugacy classes of
involutions in $\Aut(\fru_0)$, and the corresponding real forms are
$\mathfrak{so}(p,2n-p),\mathfrak{so^{\ast}}(2n,\bbR)$.

\smallskip

v) Type $\bf E_6$. For $\fru_0=\fre_6$, let $\tau$ be a specific diagram involution
defined by
\begin{eqnarray*} &&\tau(H_{\alpha_1})=H_{\alpha_6},
\tau(H_{\alpha_6})=H_{\alpha_1}, \tau(H_{\alpha_3})=H_{\alpha_5},
\tau(H_{\alpha_5})=H_{\alpha_3},
\\&& \tau(H_{\alpha_2})=H_{\alpha_2},
\tau(H_{\alpha_4})=H_{\alpha_4},
\tau(X_{\pm{\alpha_1}})=X_{\pm{\alpha_6}},
\tau(X_{\pm{\alpha_6}})=X_{\pm{\alpha_1}},\\&&
\tau(X_{\pm{\alpha_3}})=X_{\pm{\alpha_5}},
\tau(X_{\pm{\alpha_5}})=X_{\pm{\alpha_3}},
\tau(X_{\pm{\alpha_2}})=X_{\pm{\alpha_2}},
\tau(X_{\pm{\alpha_4}})=X_{\pm{\alpha_4}}.
\end{eqnarray*} Let \[\sigma_1=\exp(\pi i H'_2), \sigma_2=\exp(\pi i
(H'_1+H'_6)), \sigma_3=\tau, \sigma_4=\tau \exp(\pi i H'_2).\] Then
$\sigma_1,\sigma_2,\sigma_3,\sigma_4$ represent all conjugacy
classes of involutions in $\Aut(\fru_0)$, which correspond to
Riemannian symmetric spaces of type {\bf EII, EIII, EIV, EI} and
the corresponding real forms are
$\fre_{6,-2},\fre_{6,14},\fre_{6,26},\fre_{6,-6}$.
$\sigma_1,\sigma_2$ are inner automorphisms, $\sigma_3,\sigma_4$
are outer automorphisms.

\smallskip

vi) Type $\bf E_7$. For $\fru_0=\fre_7$, let \[\sigma_1=\exp(\pi i H'_2),
\sigma_2=\exp(\pi i \frac{H'_2+H'_5+H'_7}{2}), \sigma_3=\exp(\pi i
\frac{H'_2+H'_5+H'_7+2H'_1}{2}).\] Then $\sigma_1,\sigma_2,\sigma_3$
represent all conjugacy classes of involutions in $\Aut(\fru_0)$,
which correspond to Riemannian symmetric spaces of type {\bf EVI, EVII, EV}
and the corresponding real forms are
$\fre_{7,3},\fre_{7,25},\fre_{7,-7}$.

\smallskip

vii) Type $\bf E_8$. For $\fru_0=\fre_8$, let \[\sigma_1=\exp(\pi i H'_2),
\sigma_2=\exp(\pi i (H'_2+H'_1)).\] Then $\sigma_1,\sigma_2$
represent all conjugacy classes of involutions in $\Aut(\fru_0)$,
which correspond to Riemannian symmetric spaces of type {\bf EIX, EVIII}
and the corresponding real forms are $\fre_{8,24},\fre_{8,-8}$.

\smallskip

viii) Type $\bf F_4$. For $\fru_0=\frf_4$, let \[\sigma_1=\exp(\pi i H'_1),
\sigma_2=\exp(\pi i H'_4).\] Then $\sigma_1,\sigma_2$ represent all
conjugacy classes of involutions in $\Aut(\fru_0)$, which correspond
to Riemannian symmetric spaces of type {\bf FI, FII} and the
corresponding real forms are $\frf_{4,-4},\frf_{4,20}$.

\smallskip

ix) Type $\bf G_2$. For $\fru_0=\frg_2$, let $\sigma=\exp(\pi H'_1)$, which represents
the unique conjugacy class of involutions in $\Aut(\fru_0)$,
corresponds to Riemannian symmetric space of type {\bf G} and the
corresponding real form is $\frg_{2,-2}$.

\section{Centralizer of an automorphism}\label{FPS}

In this section we prove a property of the centralizer $G^{z}$ of
an element $z$ in a complex or compact Lie group $G$.
First, we recall a theorem due to Steinberg (\cite{Carter}, Page 93-95).

\begin{prop}(\textbf{Steinberg}) \label{Steinberg}
Let $G$ be  a connected and simply connected
semi-simple complex (or compact) Lie group.   Then the centralizer
$G^{z}$ for $z \in G$ is connected.
\end{prop}

If $G$ is a group of adjoint type and $z$ is of minimal possible order
among all element in the connected component containing $z$, then $G^z$ is also of adjoint type.
For an element $x$ in a group, we write $o(x)$ for the order $x$.

\begin{prop}  \label{FPS adjoint type}
Let $\frg$ be a complex simple Lie algebra.
Suppose that the order of
 $\theta\in\Aut(\frg)$ is equal to the order of
 the coset element $\theta \Int(\frg)$ in $\Out(\frg)=\Aut(\frg)/\Int(\frg)$, i.e.,  $o(\theta)=o(\theta \Int(\frg))$.
Then $Z_{\Int(\frg)}(\Int(\frg)^{\theta}_0)=1$. 
\end{prop}

\begin{proof}
By the assumption, $\theta$ is a diagram automorphism, this means there exists 
a Cartan subalgebra $\frt$ which is stable under $\theta$ and $\theta$ maps $\Delta^{+}$ to itself. 
For any $\alpha \in \Delta$, let $\theta(X_{\alpha})=a_{\alpha} X_{\theta \alpha}$ with $a_{\alpha}\not=0$.

Let $k=o(\theta)=o(\theta \Int(\frg))$. Then for any $\alpha \in
\Delta$, \[X_{\alpha}=\theta^{k}(X_{\alpha})=(\Pi_{0\leq j \leq
k-1}a_{\theta^{j}\alpha}) X_{\theta^{k} \alpha}.\] It follows that
$\Pi_{0\leq j \leq k-1}a_{\theta^{j}\alpha}=1$.

Let $L=\Int(\frg)^{\theta}_0$, $\frs=\frt^{\theta}$, $T=\exp (\ad \frt)$ and $S=\exp(\ad \frs)$.
It is clear that $S\subset L$. 

\smallskip

We first show that $Z_{\Int(\frg)}(S)=T$. $\frt\subset Z_{\frg}(\frs)$ is clear. Suppose that
$X_{\alpha}\in Z_{\frg}(\frs)$ for some $\alpha\in\Delta^{+}$. Since $\theta^{k}=1$, we have 
$\sum_{0\leq j\leq k-1}\theta^{j}(H)\in\frt^{\theta}=\frs$ for any $H\in \frt$. Then 
$[\sum_{0\leq j\leq k-1}\theta^{j}(H), X_{\alpha}]=0$.

For any $j$, we have 
\begin{eqnarray*}
[\theta^{j}H, X_{\alpha}]&=&\theta^{j}([H, \theta^{k-j}X_{\alpha}])
=\theta^{j}((\Pi_{0\leq i \leq
k-j-1}a_{\theta^{i}\alpha})((\theta^{k-j}\alpha)H))
X_{\theta^{k-j}\alpha}\\&=&(\Pi_{0\leq i \leq
k-j-1}a_{\theta^{i}\alpha})((\theta^{k-j}\alpha)H)(\Pi_{0\leq i \leq
j-1}a_{\theta^{k-j+i}\alpha})X_{\alpha}\\&=&(\Pi_{0\leq i \leq
k-1}a_{\theta^{i}\alpha})((\theta^{k-j}\alpha)H)X_{\alpha}
\\&=&((\theta^{k-j}\alpha)H)X_{\alpha}.
\end{eqnarray*}
Hence $0=[\sum_{0\leq j\leq k-1}\theta^{j}(H),
X_{\alpha}]=((\sum_{0\leq j\leq k-1}\theta^{k-j}\alpha)H)\cdot 
X_{\alpha}$. This implies \[\sum_{0\leq j\leq k-1}\theta^{j}\alpha=0,\]
which contradicts to all $\theta^{j}\alpha$ are positive roots. 
So $Z_{\frg}(\frs)=\frt$. Since $Z_{\Int(\frg)}(S)$ is connected, so 
$Z_{\Int(\frg)}(S) =T.$

\smallskip

Now we show that $Z_{\Int(\frg)}(L)=1$. Suppose that $1\neq\tau\in Z_{\Int(\frg)}(L)$. 
By the above, we have $Z_{\Int(\frg)}(L)\subset Z_{\Int(\frg)}(S)=T$, then 
$\tau=\Exp(\ad H)$ for some $H\in \frt$. For any $\alpha \in \Delta$, $\sum_{0\leq
j\leq k-1}\theta^{j}(X_{\alpha}) \in \frg^{\theta}$ (since $\theta^{k}=1$), so
\begin{eqnarray*}
\sum_{0\leq j\leq k-1}\theta^{j}(X_{\alpha})&=&\tau(\sum_{0\leq
j\leq k-1}\theta^{j}(X_{\alpha}))=\sum_{0\leq j\leq
k-1}\tau(\theta^{j}(X_{\alpha}))\\& =&\sum_{0\leq j\leq k-1}
e^{(\theta^{j}\alpha)H} \theta^{j}(X_{\alpha}).
\end{eqnarray*}
Since each $\theta^{j}(X_{\alpha})$ is of the form
$\theta^{j}(X_{\alpha})=b_{j}X_{\theta^{j}\alpha}$ for some
$b_{j}\not=1$, the last equality implies
$\tau(X_{\alpha})=X_{\alpha}$ if $\{\theta^{j}\alpha, 0\leq j\leq
k-1\}$ are distinct.

Since $\theta$ maps $\Delta^{+}$ to $\Delta^{+}$, a little combinatorial argument  
shows that those $\alpha \in \Delta$ with roots in $\{\theta^{j}\alpha,0\leq j\leq k-1\}$ 
pair-wisely different generate $\Delta$ (there are only 4 cases to check, 
the order 2 automorphism of root systems $\A_{n}$, $\D_{n}$, $\E_6$ and order three automorphism 
of root system $\D_4$). Since $\tau(X_{\alpha})=X_{\alpha}$ when  
$\{\theta^{j}\alpha, 0\leq j\leq k-1\}$ are distinct. So $\tau(X_{\alpha})=X_{\alpha}$ 
for any $\alpha \in\Delta$.  Hence $\tau=1$, which is to say 
$Z_{\Int(\frg)}(\Int(\frg)^{\theta}_0)=1$. 
\end{proof}

\begin{corollary} \label{centralizer}
Let $\fru_0$ be a compact simple Lie algebra.  If $\theta \in
\Aut(\fru_0)$ satisfies the condition $o(\theta)=o(\theta \Int(\fru_0))$, then
$Z_{\Int(\fru_0)}(\Int(\fru_0)^{\theta}_0)=1$.
\end{corollary}


\section{Symmetric subgroups of $\bf \Aut(\fru_0)$}\label{symmetric subgroup}

We retain the notations that  $\fru_0$ is a compact simple Lie algebra.
For each conjugacy classes of involutions
in $\Aut(\fru_0)$, we make a choice of the
representatives $\theta$ and determine the symmetric subgroup $\Aut(\fru_0)^{\theta}$.

If $\fru_0$ is a classical simple Lie algebra other than
$\mathfrak{so}(8))$, then we can use matrices to represent involutions $\theta$
and calculate the corresponding $\Aut(\fru_0)^{\theta}$.  In the case $\theta=\Ad(I_{4,4})\in\Aut(\mathfrak{so}(8))$,
we see that $\Int(\mathfrak{so}(8))^{\theta}=(\Sp(1)^{4}/Z') \rtimes D$, where
\[Z'=\{(\epsilon_1,\epsilon_2,\epsilon_3,\epsilon_4)|\epsilon_{i}=\pm{1},
\epsilon_1\epsilon_2\epsilon_3\epsilon_4=1\},\]  and $D\subset S_4$ is the (unique) normal
order four subgroup of $S_4$ with conjugation action on $(\Sp(1)^{4})/Z'$ by
permutations. Then we observe that there exists a subgroup $S\subset \Aut(\mathfrak{so}(8))$
projects isomorphically to $\Aut(\mathfrak{so}(8))/\Int(\mathfrak{so}(8))\cong S_3$ and contained in
$\Aut(\mathfrak{so}(8))^{\theta}$. It follows
that $\Aut(\mathfrak{so}(8))^{\theta}=(\Sp(1)^{4})/Z' \rtimes S_4$.

\smallskip
If $\fru_0$ is exceptional, then we
first determine the symmetric subalgebras
$\frk_0=\fru_0^{\theta}$ and the highest weights of the isotropic representations
$\frp_0=\fru_0^{-\theta}$ as $\frk_0$-modules. These 
are listed in Table 1. Since the identity component $\Aut(\fru_0)_0=\Int(\fru_0)$ is of adjoint type,
any element of $\Aut(\fru_0)^{\theta}$ acting trivially on both $\frk_0$
and $\frp_0$ must be trivial.  Thus, the isomorphism types of $\frk_0$ and
the isotropic module $\frp_0$ determine $\Aut(\fru_0)^{\theta}_0$ completely.
We show the detailed consideration case by case.

\begin{table}[ht]
\caption{Symmetrcic paris (exceptional cases)}
\centering
\begin{tabular}{|c |c |c |c |c | c| c|}
\hline & $\theta$  & $\frk_0$ & $\frp$  \\ [0.3ex] \hline

{\bf EI} &$\sigma_4=\tau\exp(\pi i H'_{2})$ & $\mathfrak{sp}(4)$ & $V_{\omega_4}$    \\
\hline

{\bf EII}&$\sigma_1=\exp(\pi i H'_2)$&
$\mathfrak{su}(6)\oplus\mathfrak{sp}(1) $ &
$\wedge^{3}\bbC^{6}\otimes\bbC^{2} $  \\ \hline

{\bf EIII}&$\sigma_2=\exp(\pi i(H'_1+H'_6))$& $\mathfrak{so}(10)\oplus
i\bbR$ & $(M_{+}\otimes 1)\oplus(M_{-}\otimes\overline{1})$ \\
\hline

{\bf EIV}&$\sigma_3=\tau$& $\frf_4$  & $V_{\omega_4}$  \\
\hline

{\bf EV}&$\sigma_3=\exp(\pi i(H'_1+H'_0))$& $\mathfrak{su}(8)$  & $\wedge^{4}\bbC^{8}$ \\
\hline

{\bf EVI}&$\sigma_1=\exp(\pi i H'_2)$& $\mathfrak{so}(12)\oplus\mathfrak{sp}(1)$  & $M_{+}\otimes\bbC^{2}$ \\
\hline

{\bf EVII}&$\sigma_2=\exp(\pi i H'_0)$& $\fre_6\oplus i\bbR$  &
$(V_{\omega_1}\otimes 1)\oplus(V_{\omega_6}\otimes\overline{1}) $ \\
\hline

{\bf EVIII}&$\sigma_2=\exp(\pi i(H'_1+H'_2))$& $\mathfrak{so}(16)$   & $M_{+}$ \\
\hline

{\bf EIX}&$\sigma_1=\exp(\pi i H'_1)$& $\fre_7\oplus\mathfrak{sp}(1)$  & $V_{\omega_7}\otimes\bbC^{2}$\\
\hline

{\bf FI}&$\sigma_1=\exp(\pi i H'_1)$& $\mathfrak{sp}(3)\oplus\mathfrak{sp}(1)$  & $V_{\omega_3}\otimes\bbC^{2}$ \\
\hline

{\bf FII}&$\sigma_2=\exp(\pi i H'_4)$& $\mathfrak{so}(9)$   & $M$ \\
\hline

{\bf G} &$\sigma=\exp(\pi i H'_1)$& $\mathfrak{sp}(1)\oplus\mathfrak{sp}(1)$   & $\Sym^{3}\bbC^{2}\otimes\bbC^{2} $ \\
\hline
\end{tabular}
\end{table}

\subsection{Type $\bf E_6$}
Now $\fru_0=\fre_6$. Suppose that $\theta=\sigma_3$ or $\sigma_4$ is an outer
automorphism.  By Corollary \ref{centralizer}, any element in
$\Int(\fru_0)^{\theta}-\Aut(\fru_0)^{\theta}_0$ acts on
$\fru_0^{\theta}$ as an outer automorphism.  Note that $\fru_0\cong
\mathfrak{sp}(4)$ or $\frf_4$ which do not have any outer
automorphism. It follows that  $\Int(\fru_0)^{\theta}=\Aut(\fru_0)^{\theta}_0$ and
$\Aut(\fru_0)^{\theta}=\Aut(\fru_0)^{\theta}_0\times\langle\theta\rangle$.

\smallskip

Suppose that $\theta=\sigma_1$ or $\sigma_2$ is an inner
automorphism, let $\theta'\in\E_6$ be an involution which maps to $\theta$ under 
the covering $\E_6\longrightarrow\Int(\fre_6)$, we have 
\begin{eqnarray*}&& \Int(\fre_6)^{\theta}=\{g\in \E_6|(\theta'g\theta'^{-1})g^{-1}\in
Z(\E_6)\}/Z(\E_6),\\&& \Int(\fre_6)^{\theta}_0=\{g\in
\E_6|(\theta' g\theta'^{-1})g^{-1}=1\}/Z(\E_6) (\textrm{use Proposition \ref{Steinberg} here}). 
\end{eqnarray*} If $\{g\in\E_6|\theta(g)g^{-1}\in Z(\E_6)\}\not=\E_6^{\theta}$. 
Then there exists $g\in\E_6$ such that \[1\neq\theta'g\theta'^{-1}g^{-1}=c\in Z(E_6).\] 
Then $g\theta' g^{-1}=\theta'c^{-1}$. But $o(\theta')=2\neq 6=o(\theta' c^{-1})$. So 
$g\theta' g^{-1}\neq\theta'c^{-1}$. Then $\{g\in\E_6|\theta(g)g^{-1}\in Z(\E_6)\}=\E_6^{\theta}$ and 
so $\Int(\fre_6)^{\theta}=\Int(\fre_6)^{\theta}_0$. Both
$\sigma_1,\sigma_2$ commutes with $\tau$, so
$\Aut(\fre_6)^{\theta}=\Int(\fre_6)^{\theta}_0\rtimes\langle\tau\rangle$.
The conjugation action of $\tau$ on $\Int(\fre_6)^{\theta}_0$ is determined by its action on
$\frk_0=\fru_0^{\theta}$, and we have
\[(\fre_6^{\sigma_1})^{\tau}=\mathfrak{sp}(3)\oplus\mathfrak{sp}(1),\ 
(\fre_6^{\sigma_2})^{\tau}=\mathfrak{so}(9).\]

\subsection{Type $\bf E_7$}
Now $\fru_0=\fre_7$ and $\Aut(\fre_7)=\Int(\fre_7)$ is connected. Let
$\pi: \E_7\longrightarrow \Aut(\fre_7)$ be a 2-fold covering with $\E_7$ a compact
connected and simply connected Lie group of type $\E_7$. Let
\[\sigma'_1=\exp(\pi iH'_2), \sigma'_2=\exp(\pi i\frac{H'_2+H'_5+H'_7}{2}),
\sigma'_3=\exp(\pi i\frac{2H'_1+H'_2+H'_5+H'_7}{2})\in \E_7. \] Then
$\pi(\sigma'_{i})=\sigma_{i}$, $o(\sigma'_1)=2$, $o(\sigma'_2)=4$
and $o(\sigma'_3)=4$. One has \begin{eqnarray*}&& \Aut(\fre_7)^{\sigma_{i}}\cong
\{g\in \E_7|(g\sigma'_{i}g^{-1})\sigma'^{-1}_{i}\in Z(\E_7)\}/Z(\E_7),\\&&
\Aut(\fre_7)^{\sigma_{i}}_0\cong \{g\in \E_7|(g\sigma'_{i}g^{-1})\sigma'^{-1}_{i}=1\}/Z(\E_7),
\end{eqnarray*} where $Z(\E_7)=\langle\exp(\pi i(H'_2+H'_5+H'_7))\rangle\cong\bbZ/2\bbZ$
is the center of $\E_7$.

\smallskip

For $\theta=\sigma_1$, suppose that there exists $g\in \E_7$ such that
$(g\sigma'_{i}g^{-1})\sigma'^{-1}_{i}=\exp(\pi i(H'_2+H'_5+H'_7))$, then
$g\exp(\pi i H'_2)g^{-1}=\exp(\pi i(H'_5+H'_7))$. Then there exists $w\in W$ such that \[w(\exp(\pi i
H'_2))=\exp(\pi i(H'_5+H'_7)).\] Since $w(\exp(\pi iH'_{\alpha_2})=\exp(\pi iH'_{w(\alpha_2)})$,
we get $\exp(\pi iH'_{w(\alpha_2)})=\exp(\pi i(H'_5+H'_7))$, then
\[w(\alpha_2)\in(\alpha_5+\alpha_7)+2\span_{\bbZ}\{\alpha_1,\alpha_2,\alpha_3,\alpha_4,\alpha_5,
\alpha_6,\alpha_7\}.\] There are no roots in $(\alpha_5+\alpha_7)+2\span_{\bbZ}\{\alpha_1,\alpha_2,
\alpha_3,\alpha_4,\alpha_5,\alpha_6,\alpha_7\}$, so there are no $g\in \E_7$ such that
$(g\sigma'_{1}g^{-1})\sigma'^{-1}_{1}=\exp(\pi i(H'_2+H'_5+H'_7))$. Then
\[\{g\in \E_7|(g\sigma'_{1}g^{-1})\sigma'^{-1}_{1}\in
Z(\E_7)\}=\E_7^{\sigma'_1},\  \Aut(\fre_7)^{\sigma_1}=\Aut(\fre_7)^{\sigma_1}_0.\]

\smallskip

For $\theta=\sigma_2,\sigma_3$, let $\omega=\exp(\frac{\pi
(X_{\alpha_2}-X_{-\alpha_2})}{2}) \exp(\frac{\pi
(X_{\alpha_5}-X_{-\alpha_5})}{2}) \exp(\frac{\pi
(X_{\alpha_7}-X_{-\alpha_7})}{2})$. Then \begin{eqnarray*}&&
\omega\sigma'_2\omega^{-1}=\sigma'^{-1}_2=\sigma'_2\exp(\pi
i(H'_2+H'_5+H'_7)),\\&&
\omega\sigma'_3\omega^{-1}=\sigma'^{-1}_3=\sigma'_3\exp(\pi
i(H'_2+H'_5+H'_7))\end{eqnarray*} and $\omega^{2}=1$. Then
$\Aut(\fre_7)^{\theta}=\Aut(\fre_7)^{\theta}_0\rtimes\langle\omega\rangle$.
The conjugation action of $\omega$ on $\Int(\fre_7)^{\theta}_0$ is determined by its action on
$\frk_0=\fru_0^{\theta}$, and we have
\[(\fre_7^{\sigma_2})^{\omega}=\frf_4,\  (\fre_7^{\sigma_3})^{\omega}=\mathfrak{sp}(4).\]
$\omega$ acts on $\frh$ as $s_{\alpha_2}s_{\alpha_5}s_{\alpha_7}$.

\subsection{Type $\bf E_8,F_4,G_2$} If  $\fru_0=\fre_8,\frf_4,\frg_2$, then $\Aut(\fru_0)$
is connected and simply connected.
By Proposition \ref{Steinberg}, $\Aut(\fru_0)^{\theta}$ is connected. Then they are determined
by $\fru_0^{\theta}$ and $\frp=\frg^{-\theta}$.

\begin{table}[ht]
\caption{Symmetric pairs and symmetric subgroups} \centering
\begin{tabular}{|c |c |c |c |c |}
\hline  Type & $(\fru_0,\frk_0)$  & rank & $\theta$  &
symmetric subgroup $\Aut(\fru_0)^{\theta}$\\ [0.3ex] \hline

{\bf AI} &$(\mathfrak{su}(n),\mathfrak{so}(n))$ & n-1 & $\overline{X}$&
$(O(n)/\langle -I\rangle)\!\times\!\langle\theta\rangle$ \\
\hline

{\bf AII} &$(\mathfrak{su}(2n),\mathfrak{sp}(n))$ & n-1 &
$J_{n}\overline{X}J_{n}^{-1}$&
$(Sp(n)/\langle -I\rangle)\!\times\!\langle\theta\rangle$ \\
\hline

{\bf AIII}
$p\!<\!q$&$(\mathfrak{su}(p\!+\!q),\!\mathfrak{s}(\mathfrak{u}(p)\!+\!
\mathfrak{u}(q)))$& $p$ &
$I_{p,q}X I_{p,q}$&$(S(U(p)\!\times\!U(q))/Z_{p+q})\!\rtimes\!\langle\tau\rangle$\\
&&&&$\Ad(\tau)=\textrm{comlex conjugation}$\\
\hline

{\bf AIII}
$p\!=\!q$&$(\mathfrak{su}\!(2p),\!\mathfrak{s}(\mathfrak{u}(p)\!+\!
\mathfrak{u}(p)))$& $p$ & $I_{p,p}X I_{p,p}$&
$(S(U(p)\!\times\!U(p))/Z_{2p})\rtimes \langle \tau, J_{p}\rangle$\\
&&&&$\Ad(J_{p})(X,Y)=(Y,X)$\\
\hline

{\bf BDI} $p\!<\!q$&
$(\mathfrak{so}(p\!+\!q),\mathfrak{so}(p)+\mathfrak{so}(q))$& $p$ &
$ I_{p,q}X I_{p,q}$&
$(O(p)\times O(q))/\langle(-I_{p},-I_{q})\rangle$\\
\hline

{\bf DI} $p>4$&$(\mathfrak{so}(2p),\mathfrak{so}(p)+\mathfrak{so}(p))$&
$p$ & $I_{p,p}X
I_{p,p}$&$((O(p)\!\times\! O(p))/\langle(-I_{p},-I_{p})\rangle)\!\rtimes\! \langle J_{p}\rangle$\\
&&&&$\Ad(J_{p})(X,Y)=(Y,X)$\\
\hline

{\bf DI} $p=4$&$(\mathfrak{so}(8),\mathfrak{so}(4)+\mathfrak{so}(4))$&
$4$ & $I_{4,4}X I_{4,4}$&$((Sp(1)^{4})/Z') \rtimes S_4$ \\
&&&&$S_4$ acts by permutaions\\
\hline

{\bf DIII}&$(\mathfrak{so}(2n),\mathfrak{u}(n))$& $n$ &
$J_{n}X J_{n}^{-1}$& $(U(n)/\{\pm{I}\})\rtimes\langle I_{n,n}\rangle$\\
&&&&$\Ad(I_{n,n})=\textrm{complex\ conjugation}$\\
\hline

{\bf CI}&$(\mathfrak{sp}(n), \mathfrak{u}(n))$& $n$ &
$(\!\textbf{i}I\!)\!X\!(\!\textbf{i}I\!)^{\!-\!1}$&
$(U(n)/\{\pm{I}\}) \rtimes\langle\textbf{j}I\rangle$\\
&&&&$\Ad(\textbf{j}I)=\textrm{complex\ conjugation}$\\
\hline

{\bf CII} $p\!<\!q$&
$(\mathfrak{sp}(p\!+\!q),\mathfrak{sp}(p)\!+\!\mathfrak{sp}(q))$&
$p$ &
$I_{p,q}X I_{p,q}$&$(Sp(p)\times Sp(q))/\langle(-I_{p},-I_{q})\rangle$\\
\hline

{\bf CII}
$p\!=\!q$&$(\mathfrak{sp}(2p),\mathfrak{sp}(p)\!+\!\mathfrak{sp}(p))$&
$p$ & $I_{p,p}X I_{p,p}$& $((Sp(p)\!\times\!Sp(p)/\langle(-I_{p},-I_{p})\rangle)\!\rtimes\! \langle J_{p}\rangle$\\
&&&&$\Ad(J_{p})(X,Y)=(Y,X)$\\
\hline

{\bf EI} &($\fre_{6}$, $\mathfrak{sp}(4)$)&6 &$\sigma_4$ &
$(Sp(4)/\langle-1\rangle)\times \langle\theta\rangle$ \\
\hline

{\bf EII}&($\fre_{6}$,
$\mathfrak{su}(6)\!+\!\mathfrak{sp}(1)$)&4&$\sigma_1$
&$(\!SU\!(6)\!\!\times\!\!Sp(1)\!/\!\langle(e^{\frac{2\pi
i}{3}}\!I\!,\!1\!),\!(\!-\!I\!,\!-\!1\!)\rangle\!)
\!\rtimes\!\langle\tau\rangle$ \\
&&&&$\frk_0^{\tau}=\mathfrak{sp}(3)\oplus\mathfrak{sp}(1)$\\
\hline

{\bf EIII}&($\fre_{6}$, $\mathfrak{so}(10)+i\bbR$) &2&$\sigma_2$&
$(\Spin(10)\times U(1)/\langle(c,i)\rangle)\rtimes\langle\tau\rangle$\\
&&&&$\frk_0^{\tau}=\mathfrak{so}(9)$\\
\hline

{\bf EIV}&($\fre_{6}$, $\frf_4$)&$2$&$\sigma_3$& $F_4\times\langle\theta\rangle$\\
\hline

{\bf EV}&($\fre_{7}$, $\mathfrak{su}(8)$)&7&$\sigma_3$&
$(SU(8)/\langle iI\rangle)\rtimes\langle\omega\rangle$\\
&&&&$\frk_0^{\omega}=\mathfrak{sp}(4)$\\
\hline

{\bf EVI}&($\fre_{7}$,
$\mathfrak{so}(12)+\mathfrak{sp}(1)$)&$4$&$\sigma_1$&
$(\Spin(12)\!\times \!Sp(1))/\langle(c,1),(-1,-1)\rangle$\\
\hline

{\bf EVII}&($\fre_{7}$, $\fre_6+i\bbR$)&3&$\sigma_2$& $((E_{6}\times
U(1))/\langle (c,e^{\frac{2\pi i}{3}})\rangle)
\rtimes\langle\omega\rangle$\\
&&&&$\frk_0^{\omega}=\frf_4$\\
\hline

{\bf EVIII}&($\fre_{8}$, $\mathfrak{so}(16)$)&$8$&$\sigma_2$&
$\Spin(16)/\langle c\rangle$\\
\hline

{\bf EIX}&($\fre_{8}$, $\fre_7+\mathfrak{sp}(1)$)&4&$\sigma_1$
& $E_{7}\times Sp(1)/\langle(c,-1)\rangle$\\
\hline

{\bf FI}&($\frf_{4}$,
$\mathfrak{sp}(3)+\mathfrak{sp}(1)$)&$4$&$\sigma_1$&$(Sp(3)\times Sp(1))/\langle(-I,-1)\rangle$\\
\hline

{\bf FII}&($\frf_{4}$, $\mathfrak{so}(9)$)&$1$&$\sigma_2$&$\Spin(9)$\\
\hline

{\bf G} &($\frg_{2}$, $\mathfrak{sp}(1)+\mathfrak{sp}(1)$)&2&
$\sigma$&$(Sp(1)\times Sp(1))/\langle(-1,-1)\rangle$\\
\hline
\end{tabular}
\end{table}

\newpage

\section{Klein four subgroups of $\bf \Aut(\fru_0)$}  \label{Klein four}

In this section, we
classify Klein four subgroups $\Gamma$ in $\Aut(\fru_0)$ up to conjugation.
We also determine the point subgroup $\Aut(\fru_0)^{\Gamma}$.
Note that such a $\Gamma$ is equal to $\{1,\theta,\sigma,\theta\sigma\}$
for two commuting involutions $\theta\neq\sigma$.
Fix an involution $\theta$, the conjugacy classes of $\Gamma$ containing
$\theta$ is determined by the conjugacy classes of involutions $\sigma\neq \theta$ in
$\Aut(\fru_0)^{\theta}$.

\subsection{Ordered commuting pairs of involutions and semisimple symmetric pairs} \label{symmetric involutions}

For a compact simple Lie algebra $\fru_0$ and its complexification $\frg$, the isomorphism 
classes of semisimple symmetric pairs $(\frg_0,\frh_0)$ with $\frg_0$ a real from of $\frg$ are 
in one-one correspondence with the conjugacy classes of ordered commuting pairs of 
(distinct) involutions $(\theta,\sigma)$ in $\Aut(\fru_0)$. When $\theta$ is fixed, the conjugacy 
classes of the pairs $(\theta,\sigma)$ in $\Aut(\fru_0)$ are in one-one correspondence with 
the $\Aut(\fru_0)^{\theta}$-conjugacy classes of involutions in $\Aut(\fru_0)^{\theta}-\{\theta\}$.

For $\fru_0$ an exceptional compact simple Lie algebra, and any representative $\theta$ of
involution classes in Subsection 2.3, we give the representatives of classes of
involutions in $\Aut(\fru_0)^{\theta}-\{\theta\}$ and identify their classes in $\Aut(\fru_0)$. 
For any classical compact simple Lie algebra $\fru_0$ and representative $\theta$ of
involution class, we have a similar classification of involutions in 
$\Aut(\fru_0)^{\theta}-\{\theta\}$, we omit it here and remark that the representatives 
can be constructed from Table 3 (Klein four subgroups). This gives a new proof to Berger 's 
classification of semisimple symmetric pairs. 

\smallskip

In most cases the symmetric subgroup $\Aut(\fru_0)^{\theta}$ is a product of classical groups 
with some twisting, for which we can classify their involution classes by matrix calculations. 
In the remain cases, $\fru_0^{\theta}=\frs_0\oplus\frz$ for an exceptional simple algebra 
$\frs_0$ an an abelian algebra $\frz=0,\ i\bbR, \textrm{ or }\mathfrak{sp}(1)$. We have a homomorphism 
$p:\Aut(\fru_0)^{\theta}\longrightarrow\Aut(\fru_0)$, then what we need to do is to   
classify involutions in $p^{-1}(\sigma)$ for $\sigma\in\Aut(\frs_0)$ an involution or the identity, 
which is not hard.
    
Note that for an exceptional compact simple Lie algebra $\fru_0$ the conjugacy class of an involution 
$\sigma\in\Aut(\fru_0)$ is determined by $\dim\frg^{\sigma}$ (this is an accident phenomenon observed by Helgason). 
For any $\sigma\in\Aut(\fru_0)^{\theta}-\{\theta\}$, the class of $\sigma$ in
$\Aut(\fru_0)$ is determined by $\dim \frg^{\sigma}=\dim\frk^{\sigma}+\frp^{\sigma}$ and 
$\dim\frk^{\sigma},\frp^{\sigma}$ are calculated from the class of $\sigma$ in
$\Aut(\fru_0)^{\theta}$ and the isomorphism types of $\fru_0^{\theta},\frp$.

\subsubsection{Type $\bf E_6$}
Now $\fru_0=\fre_6$. For $\theta=\sigma_1=\exp(\pi i H'_2)$, one has
\[\Aut(\fru_0)^{\sigma_1}=(\!\SU\!(6)\!\!\times\!\!\Sp(1)\!/\!
\langle(e^{\frac{2\pi i}{3}}\!I\!,\!1\!),\!(\!-\!I\!,\!-\!1\!)\rangle\!)
\!\rtimes\!\langle\tau\rangle,\ \sigma_1=(I,-1)=(-I,1),\] where
$\Ad(\tau)(X,Y)=(J_3\overline{X}J_{3}^{-1},Y)$. Then in $\Aut(\fru_0)$,
\begin{eqnarray*}
&&(\left(
\begin{array} {cc} -I_4&0 \\ 0& I_2 \end{array} \right),1) \sim
\sigma_2,\ \ (\left( \begin{array}{cc} -I_2&0 \\ 0& I_4 \end{array}
\right),1) \sim\sigma_1,\\&& (\left( \begin{array}{cc} iI_5&0 \\ 0&
-iI_1\end{array} \right),\textbf{i})\sim \sigma_2,\ \  (\left(\begin{array}{cc} iI_3&0 \\
0& -iI_3\end{array} \right), \textbf{i})\sim \sigma_1,\\&&
\tau\sim\sigma_3,\ \tau\sigma_1\sim\sigma_4,\ \tau(J_{3},\bf i)\sim\sigma_4.
\end{eqnarray*} And these elements give all representatives of conjugacy  classes of involutions in
$\Aut(\fru_0)^{\theta}-\{\theta\}$.

\smallskip

For $\theta=\sigma_2=\exp(\pi i (H'_1+H'_6))$, one has
\[\Aut(\fru_0)^{\sigma_2}=
((\Spin(10)\times\U(1))/\langle(c,i)\rangle)\rtimes\langle\tau\rangle,
\sigma_2=(-1,1)=(1,-1),\]  where $c=e_1 e_2...e_{10}$
and $\Ad(\tau)(x,z)=((e_1e_2\cdots e_9)x(e_1e_2\cdots e_9)^{-1},z^{-1})$. Then in $\Aut(\fru_0)$,
\begin{eqnarray*} &&(e_1 e_2 e_3
e_4,1)\sim \sigma_1,\ \ (e_1 e_2...e_8,1)\sim \sigma_2, \\&&
(\delta,\frac{1+i}{\sqrt{2}})\sim
\sigma_2,\ (-\delta, \frac{1+i}{\sqrt{2}} )\sim\sigma_1,\\&&
\tau\sim\sigma_3,\ \tau(e_1e_2e_3e_4,1)\sim\sigma_4,
\end{eqnarray*}
where $\delta=\frac{1+e_1e_2}{\sqrt {2}}\frac{1+e_3e_4} {\sqrt
{2}}...\frac{1+e_{9}e_{10}}{\sqrt {2}}$.
And these elements exhaust the representatives of conjugacy  classes of involutions in
$\Aut(\fru_0)^{\theta}-\{\theta\}$.

\smallskip

For $\theta=\sigma_3=\tau$, one has
\[\Aut(\fru_0)^{\sigma_3}=\F_{4}\times\langle\tau\rangle.\] Let
$\tau_1,\tau_2$ be involutions in $\F_4$ with \[\frf_4^{\tau_1}
\cong \mathfrak{sp}(3)\oplus \mathfrak{sp}(1),\ \frf_4^{\tau_2}\cong
\mathfrak{so}(9).\]  Then in $\Aut(\fru_0)$,
\begin{eqnarray*}
&&\tau_1\sim \sigma_1,\tau_2 \sim \sigma_2, \sigma_3\tau_1\sim
\sigma_4,\ \sigma_3\tau_2\sim\sigma_3.
\end{eqnarray*}
And these elements exhaust the representatives of conjugacy  classes of involutions in
$\Aut(\fru_0)^{\theta}-\{\theta\}$.

\smallskip

For $\theta=\sigma_4=\tau \exp(\pi i H'_2)$, one has
\[\Aut(\fru_0)^{\sigma_4}=(\Sp(4)/\langle -I\rangle)\times\langle\sigma_4\rangle.\]
Let $\tau_1=\textbf{i}I$, $\tau_2=\left(\begin{array}{cc} -I_2&0\\ 0& I_2\\
\end{array}  \right)$, $\tau_3=\left(\begin{array}{cc} -1&0\\ 0& I_3\\
\end{array}  \right)$.  Then in $\Aut(\fru_0)$,
\begin{eqnarray*}
&&\tau_1\sim\sigma_1,\ \tau_2\sim\sigma_2,\ \tau_3\sim\sigma_1, \\&&
\sigma_4\tau_1\sim\sigma_4,\ \sigma_4\tau_2\sim\sigma_4,\
\sigma_4\tau_3\sim \sigma_3.\end{eqnarray*}
And these elements exhaust the representatives of conjugacy  classes of involutions in
$\Aut(\fru_0)^{\theta}-\{\theta\}$.

\subsubsection{Type $\bf E_7$}
Now $\fru_0=\fre_7$.  For $\theta=\sigma_1=\exp(\pi i H'_2)$, one has
\[\Int(\fru_0)^{\sigma_1}=(\Spin(12)\times \Sp(1))/\langle(c,1),
(-1,-1)\rangle,\] where $\sigma_1=(-1,1)=(1,-1)$, $c=e_1
e_2...e_{12}$.  Set $\delta=\frac{1+e_1e_2}{\sqrt
{2}}\frac{1+e_3e_4}{\sqrt {2}}...\frac{1+e_{11}e_{12}}{\sqrt {2}}$.
Then in $\Aut(\fru_0)$,
\begin{eqnarray*}
&&(e_1 e_2 e_3 e_4, 1)\sim \sigma_1,\ (e_1 e_2, \bf{i})\sim\sigma_2,\
(e_1 e_2...e_6, \bf i)\sim\sigma_3,\\&& (\Pi, 1)\sim \sigma_2,\
(-\delta,1)\sim \sigma_3,\ (e_1\Pi e_1,\bf i)\sim \sigma_1.
\end{eqnarray*} And these elements give all representatives of all conjugacy classes of involutions in
$\Aut(\fru_0)^{\theta}-\{\theta\}$. Moreover,
\[\langle \sigma_1,(e_1 e_2 e_3 e_4, 1)\rangle\sim F_2,
\langle \sigma_1,(e_1\delta e_1,i) \rangle \sim F_1.\]

\smallskip

For $\theta=\sigma_2=\tau=\exp(\pi i \frac{H'_2+H'_5+H'_7}{2})$, one has
\[\Aut(\fru_0)^{\sigma_2}_0=((\E_{6}\times
\U(1))/\langle(c,e^{\frac{2\pi i}{3}})\rangle)\rtimes\langle\omega\rangle,\] where $c$ is a
non-trivial central element of $\E_{6}$, $o(c)=3$,
$\sigma_2=(1,-1)$ and $(\fre_6\oplus i\bbR)^{\omega}=\frf_4\oplus 0$.
Let $\tau_1,\tau_2$ be involutions in
$\E_6$ with \[\fre_6^{\tau_1}\cong \mathfrak{su}(6)\oplus \mathfrak{sp}(1),\
\fre_6^{\tau_2}\cong \mathfrak{so}(10)\oplus i\bbR.\] Then in $\Aut(\fru_0)$,
\begin{eqnarray*}
&&\tau_1\sim\sigma_1,\ \tau_2\sim\sigma_1,\\&&
\tau_1\sigma_2\sim\sigma_3,\
\tau_2\sigma_2\sim\sigma_2, \\&&
\omega\sim\sigma_2,\ \omega\eta\sim\sigma_3,\end{eqnarray*}
where $\eta\in \F_4=\E_6^{\omega}$ is an involution with
$\frf_4^{\eta}\cong\mathfrak{sp}(3)\oplus\mathfrak{sp}(1)$.
And these give all representatives of conjugacy classes of involutions in
$\Aut(\fru_0)^{\theta}-\{\theta\}$.

\smallskip

For $\theta=\sigma_3=\exp(\pi i \frac{H'_2+H'_5+H'_7+2H'_1}{2})$, one has
\[\Aut(\fru_0)^{\sigma_3}_0=(\SU(8)/\langle iI\rangle)\rtimes\langle\omega\rangle,
\sigma_3=\frac{1+i}{\sqrt{2}} I,\] where $\Ad(\omega)X=J_4\overline{X}J_4^{-1}$.
Let $\tau_1=\left(\begin{array}{cc}-I_{2}&\\&I_{6}\end{array}\right)$,
$\tau_2=\left(\begin{array}{cc}-I_{4}&\\&I_{4}\end{array}\right)$.
Then in $\Aut(\fru_0)$, \begin{eqnarray*}
&&\tau_1\sim\sigma_1,\ \tau_2\sim\sigma_1,\\&& \tau_1\sigma_3\sim\sigma_2,\
\tau_2\sigma_3\sim\sigma_3,\\&& \omega\sim\sigma_2,\ \omega\sigma_3\sim\sigma_3,\
\omega J_4\sim\sigma_3.\end{eqnarray*}
And these give all representatives of conjugacy classes of involutions in
$\Aut(\fru_0)^{\theta}-\{\theta\}$.

\subsubsection{Type $\bf E_8$}

Now $\fru_0=\fre_8$. For $\theta=\sigma_1=\exp(\pi i H'_2)$, one has
\[\Aut(\fru_0)^{\sigma_1}\cong (\E_7\times\Sp(1))/\langle(c,-1)\rangle,\] where
$\sigma_1=(1,-1)=(c,1)$. Let $\tau_1, \tau_2$ denote the elements in $\E_7$ with
$\tau_1^{2}=\tau_2^{2}=c$ and $\fre_7^{\tau_1}\cong\fre_6\oplus
i\bbR$, $\fre_7^{\tau_2}\cong \mathfrak{su}(8)$. Let  $\tau_3,\tau_4$ be
involutions in $\E_7$ such that there exist Klein Four subgroups $\Gamma,\Gamma'\subset E_7$ with
three non-identity elements in $\Gamma$ are all conjugate to $\tau_3$, three non-identity
elements in $\Gamma'$ are all conjugate to $\tau_4$, and
$\fre_7^{\Gamma}\cong \mathfrak{su}(6)\oplus (i\bbR)^{2}$,
$\fre_7^{\Gamma'}\cong \mathfrak{so}(8)\oplus (\mathfrak{sp}(1))^{3}$.
Then in $\Aut(\fru_0)$,
\[(\tau_1,\textbf{i})\sim \sigma_1,\ (\tau_2,\textbf{i}) \sim \sigma_2.\]
\[(\tau_3,1)\sim\sigma_1, (\tau_4,1)\sim\sigma_2.\]
And these give all representatives of conjugacy classes of involutions in
$\Aut(\fru_0)^{\theta}-\{\theta\}$.

\smallskip

For $\theta=\sigma_2=\exp(\pi i (H'_2+H'_1))$, one has
\[\Aut(\fru_0)^{\sigma_2}\cong \Spin(16)/\langle c\rangle,\] where
$\sigma_2=-1$, $c=e_1 e_2... e_{16}$. Let
\begin{eqnarray*}&&\tau_1=\ e_1e_2e_3e_4,\ \tau_2=e_1e_2e_3...e_8,\\&&
\tau_3=\delta= \frac{1+e_1e_2}{\sqrt{2}}\frac{1+e_3e_4}{\sqrt
{2}}...\frac{1+e_{15}e_{16}} {\sqrt {2}},\
\tau_4=-\delta.\end{eqnarray*} Then in $\Aut(\fru_0)$,
\[\tau_1\sim \sigma_1, \tau_2\sim
\sigma_2,\] \[\tau_3\sim\sigma_1, \ \tau_4 \sim\sigma_2.\]
And these give all representatives of conjugacy classes of involutions in
$\Aut(\fru_0)^{\theta}-\{\theta\}$.

\subsubsection{Type $\bf F_4$}

When $\fru_0=\frf_4$, for $\theta=\sigma_1=\exp(\pi i H'_1)$,
\[\Aut(\fru_0)^{\sigma_1}\cong\Sp(3)\times\Sp(1)/\langle (-I,-1)
\rangle,\] where $\sigma_1=(-I,1)=(I,-1)$. Let
\[\tau_1=(\left(\begin{array}{ccc}-1&0&0 \\ 0&1&0 \\0&0&1\\
\end{array} \right),1),\ \tau_2=(\left(\begin{array}{ccc}-1&0&0 \\
0&-1&0 \\ 0&0 &1 \\  \end{array} \right),1),\ \tau_3=(\textbf{i}I,
\textbf{i}),\] then in $\Aut(\fru_0)$,
\[\tau_1\sim\sigma_1, \tau_2\sim\sigma_2, \tau_3\sim\sigma_1.\]
And these elements represent all involution classes in
$\Aut(\fru_0)^{\theta}-\{\theta\}$.

\smallskip

For $\theta=\sigma_2=\exp(\pi i H'_4)$, one has \[\Aut(\fru_0)^{\sigma_2}\cong \Spin(9), \sigma_2=-1.\]
Let $\tau_1=e_1e_2e_3e_4$, $\tau_2=e_1e_2e_3...e_8$. Then in $\Aut(\fru_0)$,
\[\tau_1\sim \sigma_1,\ \tau_2\sim\sigma_2.\]
And these give all representatives of conjugacy classes of involutions in
$\Aut(\fru_0)^{\theta}-\{\theta\}$.

\subsubsection{Type $\bf G_2$}
When $\fru_0=\frg_2$ and $\theta=\sigma=\exp(\pi i H'_1)$, one has
\[\Aut(\fru_0)^{\sigma_1}\cong\Sp(1)\times\Sp(1)/\langle (-1,-1)
\rangle,\] where $\sigma_1=(-1,1)=(1,-1)$. Denote $\tau=(\textbf{i},
\textbf{i})$.  Then in $\Aut(\fru_0)$, $\tau\sim\sigma$.  $\tau$ represents
the unique cojugacy class of involutions in
$\Aut(\fru_0)^{\theta}-\{\theta\}$.

\medskip

In the above, we reprove Berger 's classification of semisimple symmetric pairs. 

\begin{prop}\label{Berger}
There are respectively 23, 19, 8, 5, 1 isomorphism classes of (non-trivial, i.e., $\frh_0\neq\frg_0$) 
semisimple symmetric pairs $(\frg_0,\frh_0)$ with $(\frg_0)\otimes_{\bbR}\bbC$ a complex simple 
Lie algebra of type $\bf E_6,\bf E_7, \bf E_8, \bf F_4, \bf G_2$.  
\end{prop}

\subsection{Klein four subgroups, their centralizers and symmetric pairs}

For a Klein four group $\Gamma\subset\Aut(\fru_0)$, we call the distribution of conjugacy classes of 
involutions in $\Gamma$ the {\it involution type} of $\Gamma$, we call the distribution of the 
classes of Riemannian symmetric pairs corresponding with involutions in $\Gamma$ the 
{\it symmetric space type} of $\Gamma$. Since these two types have one-one correspondence,  
so we just call the type of $\Gamma$ to mean its involution type or symmetric space type.  

\smallskip

For a compact simple Lie algebra $\fru_0$, a Klein four subgroup of $\Aut(\fru_0)$ is called 
{\it regular} if any two distinct conjugate (in $\Aut(\fru_0)$) elements 
$\sigma,\theta\in\Gamma$ is conjugate by an element $g\in\Aut(\fru_0)$ commuting with $\theta\sigma$
 (i.e. $g\in\Aut(\fru_0)^{\theta\sigma}$). 

\smallskip

A Klein four subgroup $\Gamma\subset\Aut(\fru_0)$ is called {\it special} if there are two (distinct) 
elements of $\Gamma$ which are conjugate in $\Aut(\fru_0)$, it is called {\it very special} if 
three involutions of $\Gamma$ are pair-wisely conjugate in $\Aut(\fru_0)$, otherwise it is called 
non-special. The definition of special is due to \cite{Oshima-Sekiguchi}. 

\smallskip

In Table 3 and Table 4, we list some Klein four subgroups $\Gamma_{i}\subset\Aut(\fru_0)$ for
each compact simple Lie algebra $\fru_0$ together with their symmetric space types 
(when $\fru_0$ is classical) or involution types (when $\fru_0$ is exceptional). These subgroups are 
not conjugate to each other since their fixed point subalgebra $\fru_0^{\Gamma_{i}}$ are non-isomorphic.
In the last column we also indicate whether they are special, where N=non-special, 
S=special but not very special, V= very special. Since some rows of Table 3 is a family of Klein four 
subgroups rather than a single one, in this case we use NS, SV, NSV (with the obvious meaning) to 
denote the speciality of subgroups in them. 


\begin{table}[ht]
\caption{Klein four subgroups in $\Aut(\fru_0)$ for classical cases}
\centering
\begin{tabular}{|c |c |c |c |c}
\hline
$\fru_0$ & $\Gamma_{i}$ & $\frl_0=\fru_0^{\Gamma_{i}}$ & Type\\ [0.5ex]
\hline

$\mathfrak{su}(p+q)$& $\Gamma_{p,q}=\langle \tau, I_{p,q}\rangle$ &
$\mathfrak{so}(p)+\mathfrak{so}(q)$ & \textbf{AI-AI-AIII}, NS\\
\hline

$\mathfrak{su}(2p)$& $\Gamma_{p}=\langle \tau, J_{p}\rangle$ &
$\mathfrak{u}(p)$ & \textbf{AI-AII-AIII}, N \\
\hline

$\mathfrak{su}(2p+2q)$& $\Gamma'_{p,q}=\langle\tau
J_{p+q},I'_{p,q}\rangle$ & $\mathfrak{sp}(p)+\mathfrak{sp}(q)$ &
\textbf{AII-AII-AIII}, NS\\
\hline

$\mathfrak{su}(p+q+r+s)$& $\Gamma_{p,q,r,s}$ &
$\mathfrak{s}(\mathfrak{u}(p)+\mathfrak{u}(q)+\mathfrak{u}(r)+\mathfrak{u}(s))$
& \textbf{AIII-AIII-AIII}, NSV\\
\hline

$\mathfrak{su}(2p)$& $\Gamma_{p}=\langle I_{p,p}, J_{p}\rangle$ &
$\mathfrak{su}(p)$ & \textbf{AIII-AIII-AIII}, V \\
\hline

$\mathfrak{so}(p+q+r+s)$& $\Gamma_{p,q,r,s}$ &
$\mathfrak{so}(p)+\mathfrak{so}(q)+\mathfrak{so}(r)+\mathfrak{so}(s))$
& \textbf{BDI-BDI-BDI}, NSV \\
\hline

$\mathfrak{so}(2p)$& $\Gamma_{p}=\langle J_{p}, I_{p,p}\rangle$ &
$\mathfrak{so}(p)$ & \textbf{DI-DI-DIII}, S\\
\hline

$\mathfrak{so}(2p+2q)$& $\Gamma_{p,q}=\langle J_{p+q}, I'_{p,q}\rangle$ &
$\mathfrak{u}(p)+\mathfrak{u}(q)$ & \textbf{DI-DIII-DIII}, S\\
\hline

$\mathfrak{so}(4p)$& $\Gamma'_{p}=\langle J_{2p},K_{p}\rangle$ & $\mathfrak{sp}(p)$ & 
\textbf{DIII-DIII-DIII}, V\\
\hline

$\mathfrak{sp}(p)$& $\Gamma_{p}=\langle \textbf{i}I,\textbf{j}I\rangle$ &
$\mathfrak{so}(p)$ & \textbf{CI-CI-CI}, V\\
\hline

$\mathfrak{sp}(p+q)$& $\Gamma_{p,q}=\langle \textbf{i}I, I_{p,q}\rangle$
& $\mathfrak{u}(p)+\mathfrak{u}(q)$ & \textbf{CI-CI-CII}, S\\
\hline

$\mathfrak{sp}(2p)$& $\Gamma'_{p}=\langle \textbf{i}I,
\textbf{j}J_{p}\rangle$ & $\mathfrak{sp}(p)$ & \textbf{CI-CII-CII}, S\\
\hline

$\mathfrak{sp}(p+q+r+s)$& $\Gamma_{p,q,r,s}$ &
$\mathfrak{sp}(p)+\mathfrak{sp}(q)+\mathfrak{sp}(r)+\mathfrak{sp}(s)$
& \textbf{CII-CII-CII}, NSV\\
\hline

\end{tabular}
\end{table}

\begin{table}[ht]
\caption{Klein four subgroups in $\Aut(\fru_0)$ for exceptional case}
\centering
\begin{tabular}{|c |c |c |c |c}
\hline $\fru_0$ & $\Gamma_{i}$ & $\frl_0=\fru_0^{\Gamma_{i}}$  & Type  \\ [0.5ex]
\hline
$\fre_6$& $\Gamma_{1}=\langle\exp(\pi i H'_2), \exp(\pi i H'_4)\rangle$ &
$(\mathfrak{su}(3))^{2}\oplus (i\bbR)^{2}$ & $(\sigma_1,\sigma_1,\sigma_1)$, V \\
\hline

 $\fre_6$& $\Gamma_{2}=\langle\exp(\pi i H'_4), \exp(\pi i (H'_3+H'_4+H'_5))\rangle$
 & $\mathfrak{su}(4) \oplus(\mathfrak{sp}(1))^{2}\oplus i \bbR$&  $(\sigma_1,\sigma_1,\sigma_2)$, S \\
\hline

$\fre_6$& $\Gamma_{3}=\langle\exp(\pi i (H'_2+H'_1)), \exp(\pi i
(H'_4+H'_1))\rangle$ & $\mathfrak{su}(5) \oplus (i\bbR)^{2}$ &
$(\sigma_1,\sigma_2,\sigma_2)$, S  \\
\hline

$\fre_6$& $\Gamma_{4}=\langle\exp(\pi i (H'_1+H'_6)), \exp(\pi i
(H'_3+H'_5))\rangle$ &  $\mathfrak{so}(8)\oplus (i\bbR)^{2}$  &
$(\sigma_2,\sigma_2,\sigma_2)$, V \\
\hline

$\fre_6$& $\Gamma_{5}=\langle\exp(\pi i H'_2), \tau \rangle$ &
$\mathfrak{sp}(3) \oplus \mathfrak{sp}(1)$  &
$(\sigma_1,\sigma_3,\sigma_4)$, N \\
\hline

$\fre_6$& $\Gamma_{6}=\langle\exp(\pi i H'_2), \tau \exp(\pi i
H'_4)\rangle$ & $\mathfrak{so}(6)\oplus i\bbR$  &
$(\sigma_1,\sigma_4,\sigma_4)$, S \\
\hline

$\fre_6$& $\Gamma_{7}=\langle\exp(\pi i(H'_1+H'_6))), \tau\rangle$  &
$\mathfrak{so}(9)$ & $(\sigma_2,\sigma_3,\sigma_3)$, S \\
\hline

$\fre_6$& $\Gamma_{8}=\langle\exp(\pi i(H'_1+H'_6)), \tau \exp(\pi i
H'_2)\rangle$ &  $\mathfrak{so}(5)\oplus \mathfrak{so}(5)$
& $(\sigma_2,\sigma_4,\sigma_4)$, S \\
\hline

$\fre_7$& $\Gamma_{1}=\langle\exp(\pi i H'_2), \exp(\pi i H'_4)\rangle$ &
$\mathfrak{su}(6)\oplus (i\bbR)^{2}$&$(\sigma_1,\sigma_1,\sigma_1)$, V\\
\hline

$\fre_7$& $\Gamma_{2}=\langle\exp(\pi i H'_2), \exp(\pi i H'_3)\rangle$ &
$\mathfrak{so}(8)\oplus (\mathfrak{sp}(1))^{3}$  &
$(\sigma_1,\sigma_1,\sigma_1)$, V\\
\hline

$\fre_7$& $\Gamma_{3}=\langle\exp(\pi i H'_2), \tau \rangle$ &
$\mathfrak{so}(10) \oplus (i\bbR)^{2}$   &
$(\sigma_1,\sigma_2,\sigma_2)$, S \\
\hline

$\fre_7$& $\Gamma_{4}=\langle\exp(\pi i H'_1), \tau \rangle$  &
$\mathfrak{su}(6)\oplus \mathfrak{sp}(1) \oplus i\bbR$   &
$(\sigma_1,\sigma_2,\sigma_3)$, N  \\
\hline

$\fre_7$& $\Gamma_{5}=\langle\exp(\pi i H'_2), \tau \exp(\pi i
H'_1)\rangle$ & $\mathfrak{su}(4)\oplus \mathfrak{su}(4) \oplus
i\bbR$ & $(\sigma_1,\sigma_3,\sigma_3)$, S \\
\hline

$\fre_7$& $\Gamma_{6}=\langle\tau, \omega\rangle$ & $\frf_4$ &
$(\sigma_2,\sigma_2,\sigma_2)$, V \\
\hline

$\fre_7$& $\Gamma_{7}=\langle\tau, \omega \exp(\pi i H'_1)\rangle$ &
$\mathfrak{sp}(4)$ &  $(\sigma_2,\sigma_3,\sigma_3)$, S \\
\hline

$\fre_7$& $\Gamma_{8}=\langle\tau \exp(\pi i H'_1), \omega \exp(\pi i
H'_3)\rangle$ & $\mathfrak{so}(8)$ &
$(\sigma_3,\sigma_3,\sigma_3)$, V \\
\hline

$\fre_8$& $\Gamma_{1}=\langle\exp(\pi i H'_2), \exp(\pi i H'_4)\rangle$ &
$\fre_6\oplus (i\bbR)^{2}$  &  $(\sigma_1,\sigma_1,\sigma_1)$, V\\
\hline

$\fre_8$& $\Gamma_{2}=\langle\exp(\pi i H'_2), \exp(\pi i H'_1)\rangle$ &
$\mathfrak{so}(12)\oplus(\mathfrak{sp}(1))^{2}$ &$(\sigma_1,\sigma_1,\sigma_2)$, S\\
\hline

$\fre_8$& $\Gamma_{3}=\langle\exp(\pi i H'_2), \exp(\pi i
(H'_1+H'_4))\rangle$
& $\mathfrak{su}(8)\oplus i\bbR$ & $(\sigma_1,\sigma_2,\sigma_2)$, S\\
\hline

$\fre_8$& $\Gamma_{4}=\langle\exp(\pi i (H'_2+H'_1)),\exp(\pi i
(H'_5+H'_1))\rangle$ & $\mathfrak{so}(8)\oplus \mathfrak{so}(8)$ &
$(\sigma_2,\sigma_2,\sigma_2)$, V\\
\hline

$\frf_4$& $\Gamma_{1}=\langle\exp(\pi i H'_2),\exp(\pi i H'_1)\rangle$ &
$\mathfrak{su}(3)\oplus (i\bbR)^{2}$ &  $(\sigma_1,\sigma_1,\sigma_1)$, V\\
\hline

$\frf_4$& $\Gamma_{2}=\langle\exp(\pi i H'_3),\exp(\pi i H'_2)\rangle$ &
$\mathfrak{so}(5)\oplus(\mathfrak{sp}(1))^{2}$ &
$(\sigma_1,\sigma_1,\sigma_2)$, S\\
\hline

$\frf_4$& $\Gamma_{3}=\langle\exp(\pi i H'_4),\exp(\pi i H'_3)\rangle$ &
$\mathfrak{so}(8)$& $(\sigma_2,\sigma_2,\sigma_2)$, V\\
\hline

$\frg_2$& $\Gamma=\langle\exp(\pi i H'_1), \exp(\pi i H'_2)\rangle$ &
$(i\bbR)^{2}$ &  $(\sigma,\sigma,\sigma)$, V \\
\hline

\end{tabular}
\end{table}

\begin{theorem}\label{T}
For a compact simple Lie algebra $\fru_0$, any Klein four subgroup $\Gamma\subset\Aut(\fru_0)$ 
is conjugate to one in Table 3 or Table 4 and they are all regular. 
\end{theorem}

\begin{proof}
When $\fru_0$ is a classical simple Lie algebra, we can do matrix calculation to show Table 3 
is complete. When $\fru_0$ is an exceptional simple Lie algebra, from the Klein four subgroups, 
we get non-conjugate commuting pairs of involutions $(\theta_1,\theta_2)$ distinguished by 
the isomorphism type of $\fru_0^{\langle\theta_1,\theta_2\rangle}$ or the distribution of the 
classes of (ordered) $\theta_1,\theta_2,\theta_3$.  When $\fru_0$ is of type 
$\bf E_6,\bf E_7,\bf E_8,\bf F_4,\bf G_2$, we get (at least) 23,19,8,5,1 non-conjugate 
commuting pairs respectively. By Proposition \ref{Berger}, they represent all conjugacy classes of 
commuting pairs of involutions. So table 4 is complete. 

For an exceptional simple Lie algebra $\fru_0$, suppose that some Klein four subgroup fails to be regular, 
then we can construct non-conjugate commuting pairs 
$(\theta_1,\theta_2)$ and $(\theta'_1,\theta'_2)$ with $\langle\theta_1,\theta_2\rangle=
\langle\theta'_1,\theta'_2\rangle$, $\theta_1\sim\theta'_1,\theta_2\sim\theta'_2,
\theta_1\theta_2\sim\theta'_1\theta'_2$. Then there should exist more isomorphism classes of 
symmetric pairs, which is not the case. So any Klein four subgroup is regular. 

In general, to show all Klein four subgroups of $\Aut(\fru_0)$ are regular, we just 
need to check for any commuting pair of involutions $\theta_1,\theta_2\in\Aut(\fru_0)$ with 
$\theta_1\sim\theta_2$ (in $\Aut(\fru_0)$), $\theta_1,\theta_2$ are conjugate in $\Aut(\fru_0)^{\theta}$, 
where $\theta=\theta_1\theta_2$. Fix $\theta$ as an representative in Subsection 2.3, When $\fru_0$ is 
an exceptional simple Lie algebra, this is already checked in the last subsection; when 
$\fru_0$ is a classical simple Lie algebra, we can check this from the data in Table 3 
(list of Klein fours with symmetric space type) and Table 2 (symmetric subgroups). 
\end{proof}

An equivalent statement of the second statement in Theorem \ref{T} (any Klein four subgroup is regular) is, 
two commuting pairs of involtutions $(\theta,\sigma)$ and 
$(\theta',\sigma')$ are conjugate in $\Aut(\fru_0)$ if and only if 
\[\theta\sim\theta', \sigma\sim\sigma',\theta\sigma\sim\theta'\sigma'\]
and the Klein four subgroups $\langle\theta,\sigma\rangle, \langle\theta',\sigma'\rangle\sim$ are conjugate. 
This statement imples the second statemnt in Theorem \ref{T} is clear. To derive this statement from 
Theorem \ref{T}, given two pairs $(\theta,\sigma)$ and $(\theta',\sigma')$ with 
$\theta\sim\theta'$, $\sigma\sim\sigma'$, $\theta\sigma\sim\theta'\sigma'$ and 
$\langle\theta,\sigma\rangle\sim\langle\theta',\sigma'\rangle\sim$.  
After replace $(\theta',\sigma')$ by a pair conjugate to it, we may assume 
$\langle\theta,\sigma\rangle=\langle\theta',\sigma'\rangle\sim$, that is, $(\theta,\sigma),(\theta',\sigma')$ 
generate a same Klein four subgroup $F$. Then by Theorem \ref{T} they are conjugate.  Since any Klein 
subgroup of $\Aut(\fru_0)$ is regular, a conjugacy class of Klein four subgroups gives 6, 3, 1 isomorphism types 
of semisimple symmetric pairs when it is non-special, special but not very special, very special respectively.

The fact that Klein four subgroups in $\Aut(\fru_0)$ are all regular is an interesting phenomenon. 
The property regular can be generalized to any closed subgroup of any Lie group, there are vast of examples 
of non-regular subgroups given in \cite{Larsen}.  


\smallskip

From Table 4 and 6, we observe the following statements.

\begin{prop}
When $\fru_0$ is an exceptional compact simple Lie algebra, any two classes of involutions 
have commuting representatives; for any Klein four group $\Gamma\subset\Aut(\fru_0)$ the 
centralizer $\Aut(\fru_0)^{\Gamma}$ meets with all connected components of $\Aut(\fru_0)$. 
\end{prop}

When $\fru_0$ is a classical simple Lie algebra, both statements of the above proposition fail in general. 
For example, in $\Aut(\mathfrak{su}(2n))$ and for an odd $p$ with $1\leq p\leq n-1$, $\tau\circ\Ad(I_{n,n})$ 
($\tau=\textrm{complex conjugation}$) doesn't commute with any involution conjugate to $\Ad(I_{p,2n-p})$; 
in $\Aut(\mathfrak{so}(4n))$, $\Aut(\mathfrak{so}(4n))^{\Gamma_{n}}\subset\Int(\mathfrak{so}(4n))$ 
(see Table 3 for the definition of $\Gamma_{n}$).   

\smallskip

Lastly for each Klein four subgroup $\Gamma$ listed in Tables 3 and  4 with two generators 
$\theta,\sigma \in\Aut(\fru_0)$. Then we get the centralizer $\Aut(\fru_0)^{\Gamma}$ by calculating 
$(\Aut(\fru_0)^{\theta})^{\sigma}$. The results of $\Aut(\fru_0)^\Gamma$ are listed in Table 5 for 
the classical cases and in Tables 6 for the exceptional cases.

\begin{table}[ht]
\caption{Fixed point subgroups of Klein four subgroups: classical cases}
\centering
\begin{tabular}{|c |c |c |}
\hline
$\fru_0$ & $\Gamma_{i}$ & $L=\Aut(\fru_0)^{\Gamma_{i}}$\\ [0.5ex]
\hline

$\mathfrak{su}(p+q), p\neq q$&$\Gamma_{p,q}$& $((O(p)\times
O(q))/\langle(-I_{p},-I_{q})\rangle) \times\langle\tau\rangle$ \\
\hline

$\mathfrak{su}(2p)$&$\Gamma_{p,p}$& $((O(p)\times
O(p))/\langle(-I_{p},-I_{p})\rangle)
\rtimes\langle\tau,J_{p}\rangle$, \\
&&$\Ad(J_{p})(X,Y)=(Y,X)$, $\Ad(\tau)=1$\\
\hline

$\mathfrak{su}(2p)$&$\Gamma'_{p}$& $(U(p)/\langle-I_{p}\rangle)
\rtimes\langle \tau,z\rangle$, $\Ad(z)=1$ \\
\hline

$\mathfrak{su}(2p+2q),p\neq q$&$\Gamma'_{p,q}$& $((Sp(p)\times
Sp(q))/\langle(-I_{p},-I_{q})\rangle)
\times\langle\tau J_{p+q}\rangle$ \\
\hline

$\mathfrak{su}(4p)$&$\Gamma'_{p,p}$& $((Sp(p)\times
Sp(p))/\langle(-I_{p},-I_{p})\rangle)
\rtimes\langle\tau J_{2p},J_{p}\rangle$, \\
&&$\Ad(J_{p})(X,Y)=(Y,X)$, $\Ad(\tau J_{2p})=1$\\
\hline

$\mathfrak{su}(p+q+r+s)$& $\Gamma_{p,q,r,s}$& $((S(U(p)\times
U(q)\times U_{r}\times U_{s})/\langle Z_{p+q+r+s}\rangle)\rtimes\langle\tau\rangle$\\
&&$\Ad(\tau)=\textrm{complex conjugation}$\\
\hline

$\mathfrak{su}(2p+2r)$,$p\neq r$& $\Gamma_{p,p,r,r}$& $((S(U(p)\times
U(p)\times U_{r}\times U_{r})/\langle Z_{2p+2r}\rangle)\rtimes\langle\tau,J_{p,r}\rangle$\\
&&$\Ad(J_{p,r})(X_1,X_2,X_3,X_4)=(X_2,X_1,X_4,X_3)$\\
\hline

$\mathfrak{su}(4p)$& $\Gamma_{p,p,p,p}$& $((S(U(p)\times
U(p)\times U(p)\times U(p))/\langle Z_{4p}\rangle)\rtimes\langle\tau,J_{2p},J_{p,p}\rangle$\\
&&$\Ad(J_{2p})(X_1,X_2,X_3,X_4)=(X_3,X_4,X_1,X_2)$\\
\hline

$\mathfrak{su}(2p)$& $\Gamma_{p}$ & $PSU(p)\rtimes\langle F_{p},\tau\rangle$ \\
&&$\Ad(\tau)=\textrm{complex conjugation}$, $\Ad(F_{p})=1$\\
\hline

$\mathfrak{so}(p+q+r+s)$& $\Gamma_{p,q,r,s}$& $(O(p)\times O(q)\times
O(r)\times O(s))/\langle-I_{p+q+r+s}\rangle$\\ \hline

$\mathfrak{so}(2p+2r)$,$p\neq r$& $\Gamma_{p,p,r,r}$& $((O(p)\times
O(p)\times O(r)\times O(r))/\langle-I_{2p+2r}\rangle))\rtimes
\langle J_{p,r}\rangle$\\
&&$\Ad(J_{p,r})(X_1,X_2,X_3,X_4)=(X_2,X_1,X_4,X_3)$\\
\hline

$\mathfrak{so}(4p)$, $p\neq 2$& $\Gamma_{p,p,p,p}$&
$((O(p))^{4}/\langle -I_{4p}\rangle)\rtimes\langle J_{2p},J_{p,p}\rangle$\\
&&$\Ad(J_{2p})(X_1,X_2,X_3,X_4)=(X_3,X_4,X_1,X_2)$\\
\hline

$\mathfrak{so}(8)$& $\Gamma_{2,2,2,2}$& $(U(1)^{4}/Z')\rtimes\langle
\epsilon_{1,2},\epsilon_{1,3},\epsilon_{1,4},S_{4}\rangle$\\
&&$\Ad(\epsilon_{1,2})(X_1,X_2,X_3,X_4)=(-X_1,-X_{2},X_3,X_4)$,etc\\
&&$S_4$ acts by permutations\\
\hline

$\mathfrak{so}(2p)$& $\Gamma_{p}$ & $(O(p)/\langle-I_{p}\rangle)\times F_{p}$\\
\hline

$\mathfrak{so}(2p+2q), p\neq q$&$\Gamma_{p,q}$& $((U(p)\times
U(q))/\langle(-I_{p},-I_{q})\rangle)\rtimes\langle\tau\rangle$\\
&&$\Ad(\tau)=\textrm{complex conjugation}$\\
\hline

$\mathfrak{so}(4p)$&$\Gamma_{p,p}$& $((U(p)\times
U(p))/\langle(-I_{p},-I_{p})\rangle)
\rtimes\langle\tau,J_{p}\rangle$, \\
&&$\Ad(J_{p})(X,Y)=(Y,X)$\\
\hline

$\mathfrak{so}(4p)$& $\Gamma'_{p}$ & $(Sp(p)/\langle-I_{p}\rangle)\times F'_{p}$\\
\hline

$\mathfrak{sp}(n)$& $\Gamma_{p}$ & $(O(n)/\langle-I_{n}\rangle)\times F_{p}$\\
\hline

$\mathfrak{sp}(p+q), p\neq q$&$\Gamma_{p,q}$& $((U(p)\times
U(q))/\langle(-I_{p},-I_{q})\rangle)\times\langle\tau\rangle$\\
\hline

$\mathfrak{sp}(2p)$&$\Gamma_{p,p}$& $((U(p)\times
U(p))/\langle(-I_{p},-I_{p})\rangle)\rtimes\langle\tau,J_{p}\rangle$, \\
&&$\Ad(\tau)=\textrm{complex conjugation}$, $\Ad(J_{p})(X,Y)=(Y,X)$\\
\hline

$\mathfrak{sp}(2p)$& $\Gamma'_{p}$ & $(Sp(p)/\langle-I_{p}\rangle)\times F'_{p}$\\
\hline

$\mathfrak{sp}(p+q+r+s)$& $\Gamma_{p,q,r,s}$& $(Sp(p)\times Sp(q)\times
Sp(r)\times Sp(s))/\langle-I_{p+q+r+s}\rangle$\\ \hline

$\mathfrak{sp}(2p+2r)$,$p\neq r$& $\Gamma_{p,p,r,r}$& $((Sp(p)\times
Sp(p)\times Sp(r)\times Sp(r))/\langle-I_{2p+2r}\rangle)\rtimes
\langle J_{p,r}\rangle$\\
&&$\Ad(J_{p,r})(X_1,X_2,X_3,X_4)=(X_2,X_1,X_4,X_3)$\\
\hline

$\mathfrak{sp}(4p)$& $\Gamma_{p,p,p,p}$& $((Sp(p))^{4}/\langle -I_{4p}\rangle)\rtimes\langle J_{2p},J_{p,p}\rangle$\\
&&$\Ad(J_{2p})(X_1,X_2,X_3,X_4)=(X_3,X_4,X_1,X_2)$\\
\hline
\end{tabular}
\end{table}



\begin{table}[ht]
\caption{Fixed point subgroups of Klein four subgroups: exceptional cases}
\centering
\begin{tabular}{|c |c |c |c }
\hline $\fru_0$ & $\Gamma_{i}$ & $L=\Aut(\fru_0)^{\Gamma_{i}}$\\ [0.5ex]
\hline

$\fre_6$& $\Gamma_{1}$ & $((SU(3)\times SU(3)\times U(1)\times
U(1))/\langle(e^{\frac{2\pi i}{3}} I,I,e^{\frac{2\pi
i}{3}},1),(I,e^{\frac{2\pi i}{3}} I,e^{\frac{-2\pi
i}{3}},1)\rangle)\rtimes \langle z,\tau\rangle$, \\
&& $\Ad(\tau)(X,Y,\lambda,\mu)=(\overline{Y},\overline{X},\lambda,\mu)$,
$\Ad(z)(X,Y,\lambda,\mu)=(Y,X,\lambda^{-1},\mu^{-1})$\\ \hline

$\fre_6$&$\Gamma_{2}$& $(SU(4)\times Sp(1)\times Sp(1)\times U(1))/
\langle(iI,-1,1,i),(I,-1,-1,-1)\rangle)\rtimes \langle\tau\rangle$, \\
&&$\Ad(\tau)(X,y,z,\lambda)=(J_{2}\overline{X}(J_{2})^{-1},y,z,\lambda^{-1})$
\\ \hline

$\fre_6$& $\Gamma_{3}$ & $(SU(5)\times U(1)\times U(1))\rtimes \langle
\tau' \rangle$, \\&&$\Ad(\tau')(X,\lambda,\mu)=(\overline{X},\lambda^{-1},\mu^{-1})$  \\
\hline

$\fre_6$& $\Gamma_{4}$ & $((\Spin(8)\times U(1)\times
U(1))/\langle(-1,-1,1),(c,1,-1)\rangle)\rtimes\langle\tau\rangle$,
\\ & &$\Ad(\tau)(x,\lambda,\mu)=(x,\lambda^{-1},\mu^{-1})$\\
\hline

$\fre_6$& $\Gamma_{5}$ & $((Sp(3)\times
Sp(1))/\langle(-I,-1)\rangle)\times \langle \tau\rangle$
\\ \hline

$\fre_6$& $\Gamma_{6}$ & $((SO(6)\times
U(1))/\langle(-I,-1)\rangle)\rtimes\langle\tau',z\rangle$,\\
&&$\Ad(z)(X,\lambda)=(I_{3,3}X I_{3,3},\lambda^{-1})$, $\Ad(\tau')=1$
\\ \hline

$\fre_6$& $\Gamma_{7}$&$\Spin(9)\times \langle\tau\rangle$ \\
\hline

$\fre_6$& $\Gamma_{8}$ &  $((\Spin(5)\times\Spin(5))/\langle(-1,-1)
\rangle) \rtimes\langle\tau',z\rangle$,\\&& $\Ad(z)(x,y)=(y,x)$\\
\hline

$\fre_7$& $\Gamma_{1}$ & $((SU(6)\times U(1)\times
U(1))/\langle(e^{\frac{2\pi i}{3}}I, e^{\frac{-2\pi i}{3}}, 1),
(-I,1,1)\rangle)\rtimes\langle z\rangle$,\\ &
&$\Ad(z)(X,\lambda,\mu)=(J_{3}\overline{X}
J_{3}^{-1},\lambda^{-1},\mu^{-1})$ \\
\hline

$\fre_7$& $\Gamma_{2}$ & $(\Spin(8)\times
Sp(1)^{3})/\langle(c,-1,1,1),(1,-1,-1,-1),(-1,-1,-1,1)\rangle$\\
\hline

$\fre_7$& $\Gamma_{3}$ & $((\Spin(10)\times U(1)\times U(1))/\langle
(c,i,1)\rangle)\rtimes\langle z\rangle$,\\&& $\Ad(z)(x,\lambda,\mu)=(e_1
x e_{1}^{-1},\lambda^{-1},\mu^{-1})$ \\ \hline

$\fre_7$& $\Gamma_{4}$  & $((SU(6)\times Sp(1)\times
U(1))/\langle(e^{\frac{2\pi i}{3}}I,1,e^{\frac{-2\pi i}{3}}),
(-I,-1,1)\rangle)\rtimes\langle z\rangle$,
\\&&$\Ad(z)(X,y,\lambda)=(J_{3}\overline{X}
J_{3}^{-1},y,\lambda^{-1})$\\
\hline

$\fre_7$& $\Gamma_{5}$ & $((\Spin(6)\times\Spin(6)\times U(1))/
\langle(c,c',1),(1,-1,-1)\rangle)\rtimes\langle z_1,z_2\rangle$,\\
& &$\Ad(z_1)(x,y,\lambda)=(y,x,\lambda^{-1})$,
$\Ad(z_2)(x,y,\lambda)=(e_1 x e_{1}^{-1}, e_{1} y e_{1}^{-1},\lambda^{-1})$\\
\hline

$\fre_7$& $\Gamma_{6}$ & $F_4\times \langle \tau,\omega\rangle$\\
\hline

$\fre_7$& $\Gamma_{7}$ & $(Sp(4)/\langle -I\rangle)\times\langle\tau,\omega'\rangle$\\
\hline

$\fre_7$& $\Gamma_{8}$ & $(SO(8)/\langle -I\rangle)\times\langle\tau',\omega'\rangle$\\
\hline

$\fre_8$& $\Gamma_{1}$ & $((E_6\times U(1)\times
U(1))/\langle(c,e^{\frac{2\pi i}{3}},1)\rangle)
\rtimes\langle z\rangle$,\\&& $\frl_0^{z}=\frf_4\oplus 0\oplus 0$\\
\hline

$\fre_8$& $\Gamma_{2}$ &
$(\Spin(12)\times Sp(1)\times Sp(1))/\langle(c,-1,1),(-1,-1,-1)\rangle$\\
\hline

$\fre_8$& $\Gamma_{3}$ & $((SU(8)\times
U(1))/\langle(-I,1),(iI,-1)\rangle)\rtimes \langle z\rangle$,\\&&
$\frl_0^{z}=\mathfrak{sp}(4)\oplus 0$\\
\hline

$\fre_8$& $\Gamma_{4}$ & $((\Spin(8)\times
\Spin(8))/\langle(-1,-1),(c,c)\rangle)\rtimes \langle z\rangle$, \\&&
$\Ad(z)(x,y)=(y,x)$\\
\hline

$\frf_4$& $\Gamma_{1}$ & $((SU(3)\times U(1)\times
U(1))/\langle(e^{\frac{2\pi i}{3}}I,e^{\frac{-2\pi
i}{3}},1)\rangle)\rtimes\langle z\rangle$,\\&&
$\frl_0^{z}=\mathfrak{so}(3)\oplus 0\oplus 0$\\
\hline

$\frf_4$& $\Gamma_{2}$ & $((Sp(2)\times Sp(1)\times
Sp(1))/\langle(-I,-1,-1)\rangle$\\
\hline

$\frf_4$& $\Gamma_{3}$ & $\Spin(8)$\\
\hline

$\frg_2$& $\Gamma$ &
$(U(1)\times U(1))\rtimes\langle z\rangle$,\\&& $\Ad(z)(\lambda,\mu)=(\lambda^{-1},\mu^{-1})$\\
\hline
\end{tabular}
\end{table}


\begin{thebibliography}{MO}

\bibitem[B]{Berger} M. Berger,
\emph{Les espaces sym\'etriques noncompacts}, Ann. Sci. \'Ecole Norm.
Sup. (3)  \textbf{74} (1957) 85--177.

\bibitem[C]{Carter} R. W. Carter,
\emph{Finite groups of Lie type. Conjugacy classes and complex
characters}, Reprint of the 1985 original. Wiley Classics Library. A
Wiley-Interscience Publication. John Wiley and Sons, Ltd.,
Chichester, 1993. xii+544 pp. ISBN: 0-471-94109-3 20C33 (20-02
20G40) New York-Sydney,   

\bibitem[CH]{Chuah-Huang} M.-K. Chuah and J.-S. Huang,
\emph{Double vogan diagrams and semisimple symmetric spaces}, to
appear in Transactions of AMS.


\bibitem[He]{Helgason} Helgason, S. \emph{Differential geometry, Lie groups, and 
symmetric spaces}. Pure and Applied Mathematics, 80. Academic Press, Inc. 
[Harcourt Brace Jovanovich, Publishers], New York-London, 1978. xv+628 pp. 
ISBN: 0-12-338460-5.

\bibitem[He]{Helminck} A. G. Helminck,
\emph{Algebraic groups with a commuting pair of involutions and
semisimple symmetric spaces},  Adv. in Math.  \textbf{71}  (1988),
no. 1, 21--91.

\bibitem[Hu]{Huang} J.-S. Huang,
\emph{Admissible square quadruplets and semisimple symmetric
spaces}, Adv. Math.  \textbf{165}  (2002),  no. 1, 101--123.

\bibitem[Kn]{Knapp} A. W. Knapp,
\emph{Lie groups beyond an introduction}. Second edition. Progress
in Mathematics, 140. Birkhauser Boston, Inc., Boston, MA, 2002.
xviii+812 pp. ISBN: 0-8176-4259-5 22-01

\bibitem[Ko]{Kollross}A. Kollross, \emph{Exceptional $\bbZ_2\times \bbZ_2$-symmetric
spaces}, Pacific J. Math. \textbf{242} (2009), no. 1, 113-130.

\bibitem[La]{Larsen}M. Larsen, \emph{On the conjugacy of element-conjugate 
homomorphisms.} Israel J. Math. 88 (1994), no. 1-3, 253–277.

\bibitem[L]{Lutz}R. Lutz, \emph{Sur la geometrie des espaces $\Gamma$-symmetric},
C. R. Acad. Sci. Paris Ser. \textbf{293} (1981), 55-58.


\bibitem[M]{Matsuki}T. Matsuki, \emph{Classification of two involutions on
compact semisimple Lie groups and root systems}.  J. Lie Theory  12
(2002), no. 1, 41--68.

\bibitem[OS]{Oshima-Sekiguchi} T. Oshima and J. Sekiguchi,
\emph{The restricted root system of a semisimple symmetric pair},
Group representations and systems of differential equations (Tokyo,
1982), 433--497, Adv. Stud. Pure Math. \textbf{4}, North-Holland,
Amsterdam, 1984.


\bibitem[WG]{Wolf-Gray}J. A. Wolf and A. Gray, \emph{Homogeneous spaces defined
by Lie group automorphisms, I}, J. Differential Geometry 2 (1968), 77--114.


\end{thebibliography}
\end{document}